\documentclass{amsart}
\usepackage{amssymb,amsmath,amscd,amsthm}
\usepackage{amsfonts,epsfig,latexsym,graphicx,amssymb}

\date{\today}

\newcommand{\Z}{{\mathbb Z}}
\newcommand{\R}{{\mathbb R}}
\newcommand{\C}{{\mathbb C}}
\newcommand{\N}{{\mathbb N}}
\newcommand{\T}{{\mathbb T}}
\newcommand{\Q}{{\mathbb Q}}

\newtheorem{theorem}{Theorem}[section]
\newtheorem{lemma}[theorem]{Lemma}
\newtheorem{prop}[theorem]{Proposition}
\newtheorem{coro}[theorem]{Corollary}
\newtheorem{defi}[theorem]{Definition}
\theoremstyle{definition}
\newtheorem{remark}[theorem]{Remark}
\newtheorem{exam}[theorem]{Example}
\newtheorem{quest}[theorem]{Question}

\newcommand{\tr}{\mathrm{tr} }

\def\N{{\mathbb N}}

\newcommand{\m}{{\bf m}}

\def\be{\begin{equation}}
\def\ee{\end{equation}}

\begin{document}

\title[Transport Exponents of Sturmian Hamiltonians]{Transport Exponents of Sturmian Hamiltonians}

\author[D.\ Damanik]{David Damanik}

\address{Department of Mathematics, Rice University, Houston, TX~77005, USA}

\email{damanik@rice.edu}

\thanks{D.\ D.\ was supported in part by NSF grants DMS--1067988 and DMS--1361625.}

\author[A.\ Gorodetski]{Anton Gorodetski}

\address{Department of Mathematics, University of California, Irvine, CA~92697, USA}

\email{asgor@math.uci.edu}

\thanks{A.\ G.\ was supported in part by NSF grant DMS--1301515.} % and %DMS--0901627 and IIS-1018433.}

\author[Q.-H.\ Liu]{Qing-Hui Liu}

\address{Department of Computer Science, Beijing Institute of Technology, Beijing 100081, P.R.\ China}

\email{qhliu@bit.edu.cn}

\thanks{Q.-H.\ L.\ was   supported by the National Natural Science Foundation of China, No. 11371055.}

\author[Y.-H.\ Qu]{Yan-Hui Qu}

\address{Department of Mathematics, Tsinghua University, Beijing 100084, P.R.\ China}

\email{yhqu@math.tsinghua.edu.cn}

\thanks{Y.-H.\ Q.\ was   supported by the National Natural Science Foundation of China, No. 11201256, 11371055 and 11431007.}

\begin{abstract}
We consider discrete Schr\"odinger operators with Sturmian potentials and study the transport exponents associated with them.
Under suitable assumptions on the frequency, we establish upper and lower bounds for the upper transport exponents. As an
application of these bounds, we identify the large coupling asymptotics of the upper transport exponents for frequencies of
constant type. We also bound the large coupling asymptotics uniformly from above for Lebesgue-typical frequency. A particular
consequence of these results is that for most frequencies of constant type, transport is faster than for Lebesgue almost every
frequency. We also show quasi-ballistic transport for all coupling constants, generic frequencies, and suitable phases.
\end{abstract}

\maketitle

\section{Introduction}\label{s.intro}

In this paper we study discrete Schr\"odinger operators
$$
[H_{\lambda,\alpha,\omega}\psi](n) = \psi(n+1) + \psi(n-1) + \lambda \chi_{[1 - \alpha,1)} (n \alpha + \omega \!\!\!\!\mod 1)
\psi(n)
$$
in $\ell^2(\Z)$, where $\lambda > 0$ is the coupling constant, $\alpha \in (0,1) \setminus \Q$ is the frequency, and $\omega
\in [0,1)$ is the phase. These operators are popular models of one-dimensional quasicrystals and have been studied since the
1980's; see, for example, \cite{B92, BIST89, BIT91, D07, DEG} and references therein. The special case $\alpha =
\frac{\sqrt{5}-1}{2}$ gives rise to the Fibonacci Hamiltonian, which is the most heavily studied case within this class of
operators.

By the minimality of irrational rotations of the circle and strong operator convergence it follows that the spectrum of
$H_{\lambda,\alpha,\omega}$ is independent of $\omega$ and may therefore be denoted by $\Sigma_{\lambda,\alpha}$. The spectrum
does, however, depend on $\lambda$ and $\alpha$. It is known from \cite{BIST89} that $\Sigma_{\lambda,\alpha}$ is a Cantor set
of zero Lebesgue measure. The zero-measure property implies that all spectral measures are purely singular. On the other hand,
the absence of eigenvalues was shown in \cite{DKL00} for all allowed parameters. As a consequence, the operators
$H_{\lambda,\alpha,\omega}$ have purely singular continuous spectrum.

Therefore, the RAGE Theorem (see, e.g., \cite[Theorem~XI.115]{RS3}) suggests that when studying the Schr\"odinger time
evolution for this Schr\"odinger operator, that is, $e^{-itH_{\lambda,\alpha,\omega}} \psi$ for some initial state $\psi \in
\ell^2(\Z)$, one should consider time-averaged quantities. For simplicity, let us consider initial states of the form
$\delta_n$, $n \in \Z$. Since a translation in space simply results in an adjustment of the phase, we may without loss of
generality focus on the particular case $\psi = \delta_0$. The time-averaged spreading of $e^{-itH_{\lambda, \alpha, \omega}}
\delta_0$ is usually captured on a power-law scale as follows; compare, for example, \cite{DT10, L96}. For $p > 0$, consider
the $p$-th moment of the position operator,
$$
\langle |X|^p \rangle (t) = \sum_{n \in \Z} |n|^p | \langle e^{-itH_{\lambda,\alpha,\omega}} \delta_0 , \delta_n \rangle |^2
$$
We average in time as follows. If $f(t)$ is a function of $t > 0$ and $T > 0$ is given, we denote the time-averaged function
at $T$ by $\langle f \rangle (T)$:
$$
\langle f \rangle (T) = \frac{2}{T} \int_0^{\infty} e^{-2t/T} f(t) \, dt.
$$
Then, the corresponding upper and lower transport exponents $\tilde \beta^+(p)$ and $\tilde \beta^-(p)$ are given,
respectively, by
$$
\tilde \beta^+(p) = \limsup_{T \to \infty} \frac{\log \langle \langle |X|^p \rangle \rangle (T) }{p \, \log T},
$$
$$
\tilde \beta^-(p) = \liminf_{T \to \infty} \frac{\log \langle \langle |X|^p \rangle \rangle (T) }{p \, \log T}.
$$
The transport exponents $\tilde \beta^\pm(p)$ belong to $[0,1]$ and are non-decreasing in $p$ (see, e.g., \cite{DT10}), and
hence the following limits exist:
\begin{align*}
\tilde \alpha_\ell^\pm & = \lim_{p \to 0} \tilde \beta^\pm(p), \\ \tilde \alpha_u^\pm & = \lim_{p \to \infty} \tilde
\beta^\pm(p).
\end{align*}

Sometimes it is not necessary to take time-averages, even in the purely singular continuous situation. Thus, one considers the
following analogues of the quantities above,
\begin{align*}
\beta^+(p) & = \limsup_{t \to \infty} \frac{\log \langle |X|^p \rangle (t) }{p \, \log t}, \\ \beta^-(p) & = \liminf_{t \to
\infty} \frac{\log \langle |X|^p \rangle (t) }{p \, \log t}, \\ \alpha_\ell^\pm & = \lim_{p \to 0} \beta^\pm(p), \\
\alpha_u^\pm & = \lim_{p \to \infty} \beta^\pm(p).
\end{align*}

Ballistic transport corresponds to transport exponents being equal to one, diffusive transport corresponds to the value
$\frac12$, and vanishing transport exponents correspond to (some weak form of) dynamical localization. In all other cases,
transport is called anomalous.

One-dimensional quasicrystals have long been expected to give rise to anomalous behavior. Many papers have been devoted to a
study of the transport properties of the Fibonacci Hamiltonian. For example, it is known that all the time-averaged transport
exponents defined above are strictly positive for all $\lambda > 0$, $\omega \in \T$; see \cite{DKL00}. On the other hand,
upper bounds for all the transport exponents were shown in \cite{DT07} for $\lambda > 8$. The exact large coupling asymptotics
of $\tilde \alpha_u^\pm$ were identified in \cite{DT08}, where it was shown that
\begin{equation}\label{e.inftyapproach2}
\lim_{\lambda \to \infty} \tilde \alpha_u^\pm \cdot \log \lambda = 2 \log \frac{1 + \sqrt{5}}{2},
\end{equation}
uniformly in $\omega \in \T$. In particular, the Fibonacci Hamiltonian indeed gives rise to anomalous transport for
sufficiently large coupling. The behavior in the weak coupling regime was studied in \cite{DG14}, where it was shown that
there is a constant $c > 0$ such that for $\lambda > 0$ sufficiently small, we have
$$
1 - c\lambda^2 \le \tilde \alpha_u^\pm \le 1,
$$
uniformly in $\omega \in \T$.

In this paper we will study the behavior of the transport exponents for more general frequencies $\alpha$. The main motivation
for this is that bounding transport exponents is a highly non-trivial task, especially when they take fractional values, and
very little was known beyond the Fibonacci case. The family of Sturmian Hamiltonians is the natural next step up in generality
and it is of interest to explore in which generality fractional transport exponents occur within this class and whether the
exponents (or at least their asymptotics in the regime of large or small coupling) can be identified exactly.

Given a frequency $\alpha \in (0,1) \setminus \Q$, consider its continued fraction expansion
$$
\alpha = \cfrac{1}{a_1 + \cfrac{1}{a_2 + \cfrac{1}{a_3 + \cdots}}} =: [a_1,a_2,a_3,\ldots]
$$
with uniquely determined $a_k \in \Z_+ = \{ 1, 2, 3, \ldots \}$. The $k$-th continued fraction approximant of $\alpha$ is
given by $\frac{p_k}{q_k}$, where
\begin{align*}
p_{-1} & = 1, \; p_0 = 0, \; p_{k+1} = a_{k+1} p_k + p_{k-1}, \; k \ge 0, \\ q_{-1} & = 0, \; q_0 = 1, \; q_{k+1} = a_{k+1}
q_k + q_{k-1}, \; k \ge 0.
\end{align*}

We say that $\alpha$ is of constant (resp., bounded) type if the sequence $\{ a_k \}$ is constant (resp., bounded). Moreover,
$\alpha$ is said to have bounded density if the sequence of Ces\`aro averages $\{ \frac1n \sum_{k = 1}^n a_k \}$ is bounded.
Each of these classes of $\alpha$'s has zero Lebesgue measure. For frequencies of bounded density and any coupling and phase,
it was shown in \cite{DKL00} that all time-averaged transport exponents are strictly positive. The lower bounds obtained in
\cite{DKL00}, while establishing strict positivity, are not expected to be sharp.

In this paper we establish both upper and lower bounds for the transport exponents under suitable assumptions on the
parameters. These bounds are very likely sharp in many cases (this fact has been verified in the Fibonacci case and this
verification should work in a similar way for more general frequencies). The detailed estimates can be found in
Proposition~\ref{p.transportbounds}, but the formulation of this result requires several definitions that will be given later.
Here in the introduction we focus on some particular consequences that are easy to state.

Our first result provides an upper bound for all time-averaged transport exponents in the large-coupling regime that is
uniform in the frequency on a set of full measure.

\begin{theorem}\label{t.bound}
For Lebesgue almost every $\alpha$, and uniformly in $\omega$, we have
$$
\limsup_{\lambda \to \infty} \tilde \alpha_u^+ \cdot \log \lambda \le \frac{\pi^2}{12 \log \frac{1+\sqrt{5}}{2}}.
$$
\end{theorem}

Theorem~\ref{t.bound} at first glance looks like merely a slight improvement over \cite[Corollary 1]{M10} (which implies that
$\limsup_{\lambda \to \infty} \tilde \alpha_u^+ \cdot \log \lambda \le \frac{\pi^2}{6\log 2}$), but it gives a bound that
appears to be a good candidate for a sharp bound (or even exact asymptotics). Moreover, as we will discuss below, the proofs
in \cite{M10} have several gaps, so that most of the main results of \cite{M10}, including \cite[Corollary 1]{M10}, are
actually not completely proved there. For some of the results there it is even doubtful whether they are true as stated.

Thus, recognizing that the way we prove Theorem~\ref{t.bound} suggests that the asymptotics obtained there may very well be
optimal, we would like to formulate the following question.

\begin{quest}
Is the bound given by Theorem~\ref{t.bound} sharp? Or, even stronger, is it true that for Lebesgue almost every $\alpha$, we
have $\lim_{\lambda \to \infty} \tilde \alpha_u^+ \cdot \log \lambda = \frac{\pi^2}{12 \log \frac{1+\sqrt{5}}{2}}$?
\end{quest}

Our next result identifies the precise large coupling asymptotics of the time-averaged upper transport exponents for
frequencies of constant type.

\begin{theorem}\label{t.consttype}
Let $\alpha = [m, m, m, \ldots]$, $m\ge 1$ be an irrational number of constant type. Then for every $\lambda > 4$, $\tilde
\alpha^+_u = \tilde \alpha^-_u$, the common value is independent of $\omega$, and
$$
\lim_{\lambda \to \infty} \tilde{\alpha}^{\pm}_u \cdot \log \lambda = \left\{
                                                            \begin{array}{ll}
                                                             2\log\frac{m+\sqrt{m^2+4}}{2}= 2\log{\frac{1 + \sqrt{5}}{2}} &
                                                             \hbox{if $m=1$,} \\
                                                              \log\frac{m+\sqrt{m^2+4}}{2} & \hbox{if $m\ne 1$.}
                                                            \end{array}
                                                          \right.
$$
\end{theorem}

\begin{remark}
(a) The case $m=1$ corresponds to the Fibonacci Hamiltonian, and as pointed out above, for that case Theorem \ref{t.consttype}
is not new. Namely, in that case the identity $\tilde{\alpha}_u^+ = \tilde{\alpha}_u^-$ is contained in \cite[Theorem
1.2]{DGY14}, and the asymptotics for $\lambda \to \infty$ were established in \cite{DT07, DT08}. For $m \ge 2$, the result is
new.

(b) As we will explain in Section~\ref{sec.5}, the proof of Theorem~\ref{t.consttype} extends to frequencies whose continued
fraction expansion is eventually periodic. In this case, an exact asymptotic statement as above holds, and the limit can in
principle be calculated.

(c) Theorem~\ref{t.consttype} complements the results obtained recently by Liu, Qu, Wen \cite{LQW14}, Munger \cite{M14}, and
Qu \cite{Q14} who considered general Sturmian Hamiltonians and studied the large coupling asymptotics of other interesting
quantities associated with these operators, namely the Hausdorff and box counting dimension of the spectrum, the Hausdorff
dimension of the density of states measure, and the optimal H\"older exponent of the integrated density of states.
\end{remark}

\begin{remark}
It is also interesting to compare Theorem~\ref{t.bound} and Theorem~\ref{t.consttype}. For $m \ge 6$, we have $\log
\frac{m+\sqrt{m^2+4}}{2} > \frac{\pi^2}{12 \log \frac{1+\sqrt{5}}{2}}$. Therefore, for $m \ge 6$ and sufficiently large
coupling, the corresponding transport exponent is larger than the transport exponents corresponding to Lebesque almost every
frequency. This is in some sense surprising as heuristic arguments may suggest that transport for a frequency of constant type
should not be faster than transport for a frequency with unbounded continued fraction coefficients. This, however, is one of
the mysteries of Sturmian Hamiltonians. It remains unclear in which way large continued fraction coefficients influence
dimensional and transport properties, but the pair of results above is a first step toward clarifying this.
\end{remark}

Finally, we also have a statement about the time-averaged lower transport exponent which holds for generic frequencies.

\begin{theorem}\label{t.quasiballistic2}
For every $\lambda > 0$, there exists a dense $G_\delta$ set $\mathcal{G} \subset (0,1)$ such that for every $\alpha \in
\mathcal{G}$, there exists $\omega$ such that for the operator $H_{\lambda,\alpha,\omega}$, we have $\tilde \alpha^+_\ell =
1$.
\end{theorem}

\begin{remark}\label{r.marinremark2}
(a) Both $\tilde \beta^+(p)$ and $\beta^+(p)$ are non-decreasing in $p$ and take values in $[0,1]$ \cite{DT10}. Moreover, we
have the general bounds $\tilde \beta^+(p) \le \beta^+(p)$ for every $p > 0$; see (the proof of) \cite[Lemma~7.2]{DLLY}. In
particular, for the operator $H_{\lambda,\alpha,\omega}$ found in the theorem, we have $\tilde \alpha_\ell^+ = \alpha_\ell^+ =
\tilde \alpha_u^+ = \alpha_u^+ = 1$.

(b) Theorem~6 in \cite{M10}, which is the basis for the assertion in \cite[Theorem~2]{M10}, claims that there exists $\alpha$
such that for every $\lambda > 20$, we have $\tilde \alpha_u^+ = 1$ for the operator $H_{\lambda,\alpha,0}$. What we add here
is the extension throughout the entire range of moments $p > 0$ and couplings $\lambda > 0$, the consideration of both
time-averaged and non-time-averaged quantities, and the genericity of the set of frequencies for which such a statement can be
shown. Moreover, and more importantly, there is a mistake in the proof of \cite[Theorem~6]{M10}. The author only forces
periodic approximation on a half-line, while the argument needs periodic approximation in both directions. This necessitates
the adjustments of the phase, and hence it is actually not clear whether the result as claimed in \cite{M10} even holds.

(c) While the overall argument is inspired by a construction of Last in \cite{L96}, we also borrow some ideas from the work
\cite{DLY} on quasi-ballistic transport for generic limit-periodic operators. However, there is a significant difference
between our case and smooth quasi-periodic cases covered by Last's argument and limit-periodic cases studied in \cite{DLY}. In
our case the phase dependence is delicate due to the discontinuity of the sampling function. In particular, care must be taken
in the construction of the phase for which one has the desired dynamical lower bounds, and the set of phases one obtains in
this way is significantly smaller than in the other two scenarios. We don't attempt to optimize this and only generate one
phase that works. Moreover our construction is more involved due to these difficulties.
\end{remark}

\section{The Structure and Coding of the Spectrum}\label{structure}

We describe the structure of the spectrum $\Sigma = \Sigma_{\lambda,\alpha}$ for some fixed $\lambda$ and $\alpha \in (0,1)
\setminus \Q$. We will see that $\Sigma$ has a natural covering structure which can be associated with a natural coding.

Let $k \geq 1$ and $z \in \mathbb{C}$, the transfer matrix $M_k(z)$ for zero phase over $q_k$ sites is defined by
$$
{\mathbf M}_k(z):= T_{q_k} T_{q_k - 1}  \cdots T_1,
$$
where
$$
T_j = \left[\begin{array}{cc}z - \lambda \chi_{[1 - \alpha,1)} (j \alpha \!\!\!\!\mod 1) &-1\\ 1&0\end{array}\right].
$$
By convention we take
$$
\begin{array}{l}
{\mathbf M}_{-1}(z) = \left[\begin{array}{cc} 1 & -\lambda \\ 0&1\end{array}\right],\quad {\mathbf M}_{0}(z)=
\left[\begin{array}{cc}z&-1\\ 1&0\end{array}\right].
\end{array}
$$

\smallskip

For $k\ge0$, $p\ge-1$, let $t_{(k,p)}(z) = \tr ({\mathbf M}_{k-1}(z) {\mathbf M}_k^p(z))$ and
$$
\sigma_{(k,p)} = \{ z \in \mathbb{R} : |t_{(k,p)}(z)| \leq 2 \},
$$
where $\tr M$ stands for the trace of the matrix $M$.

With these notations, we collect some known facts (see \cite{FLW11, LQW14, LW04, R97}) that will be used later.
\begin{itemize}
\item[(A)]\ Renormalization relation. For any $k \ge 0$,
$$
{\mathbf M}_{k+1}(z)={\mathbf M}_{k-1}(z)({\mathbf M}_k(z))^{a_{k+1}},
$$
so, $t_{(k+1,0)} = t_{(k-1,a_{k})}$, $t_{(k,-1)} = t_{(k-1,a_k-1)}$. We also have the recursive relations
\begin{equation}\label{e.raymondrec}
t_{(k,p+1)} = t_{(k+1,0)} t_{(k,p)} - t_{(k,p-1)}
\end{equation}
for the traces.

\item[(B)]\ Structure of $\sigma_{(k,p)}(k \ge 0, p \ge -1)$. For $\lambda > 0$, $\sigma_{(k,p)}$ is made of $\deg
    t_{(k,p)}$ disjoint closed intervals.

\item[(C)]\ Invariant. Defining $\Lambda(x,y,z) = x^2 + y^2 + z^2 - xyz - 4$, we have
\begin{equation}\label{e.invariancecond}
\Lambda(t_{(k+1,0)}, t_{(k,p)}, t_{(k,p-1)}) = \lambda^2.
\end{equation}
Thus for any $k \in \mathbb{N}$, $p \geq 0$ and $\lambda > 4$,
\begin{equation}\label{empty}
\sigma_{(k+1,0)} \cap \sigma_{(k,p)} \cap \sigma_{(k,p-1)} = \emptyset.
\end{equation}

\item[(D)]\ Covering property. For any $k \ge 0$, $p \ge -1$,
\begin{equation}\label{tao}
\sigma_{(k,p+1)} \subset \sigma_{(k+1,0)} \cup \sigma_{(k,p)},
\end{equation}
then
$$
(\sigma_{(k+2,0)} \cup \sigma_{(k+1,0)}) \subset (\sigma_{(k+1,0)} \cup \sigma_{(k,0)}).
$$
Moreover,
$$
\Sigma = \bigcap_{k \ge 0} (\sigma_{(k+1,0)} \cup \sigma_{(k,0)}).
$$
\end{itemize}

The intervals of $\sigma_{(k,p)}$ will be called the {\em bands}. When we discuss only one of these bands, it is often denoted
as $B_{(k,p)}$. Property (B) also implies $t_{(k,p)}(z)$ is monotone on $B_{(k,p)}$, and
$$
t_{(k,p)}(B_{(k,p)}) = [-2,2].
$$
We call $t_{(k,p)}$ the {\em generating polynomial} of $B_{(k,p)}$.

$\{ \sigma_{(k+1,0)} \cup \sigma_{(k,0)} : k \ge 0 \}$ form a covering of $\Sigma$. However there are some repetitions between
$\sigma_{(k,0)} \cup \sigma_{(k-1,0)}$ and $\sigma_{(k+1,0)} \cup \sigma_{(k,0)}$. It is possible to choose coverings of
$\Sigma$ elaborately such that we can get rid of these repetitions, and this is described as follows:

\begin{defi}\label{def1}
For $\lambda > 4$, $k \ge 0$, we define three types of bands as follows:

$(k,{\rm I})$-type band: a band of $\sigma_{(k,1)}$ contained in a band of $\sigma_{(k,0)}$;

$(k,{\rm II})$-type band: a band of $\sigma_{(k+1,0)}$ contained in a band of $\sigma_{(k,-1)}$;

$(k,{\rm III})$-type band: a band of $\sigma_{(k+1,0)}$ contained in a band of $\sigma_{(k,0)}$.
\end{defi}

By property (B), \eqref{empty}, and \eqref{tao}, all three kinds of types of bands are well defined, and we call these bands
{\em spectral generating bands of order $k$}. Note that for order $0$, there is only one $(0,{\rm I})$-type band
$\sigma_{(0,1)} = [\lambda - 2, \lambda + 2]$ (the corresponding generating polynomial is $t_{(0,1)} = z - \lambda$), and only
one $(0,{\rm III})$-type band $\sigma_{(1,0)} = [-2,2]$ (the corresponding generating polynomial is $t_{(1,0)} = z$). They are
contained in $\sigma_{(0,0)} = (-\infty, \infty)$ with corresponding generating polynomial $t_{(0,0)} \equiv 2$. For
convenience, we call $\sigma_{(0,0)}$ the spectral generating band of order $-1$.

\smallskip

For any $k \ge -1$, denote by $\mathcal{G}_k$ the set of all spectral generating bands of order $k$. Then the intervals in
$\mathcal{G}_k$ are disjoint. Moreover,
\begin{itemize}
\item $(\sigma_{(k+2,0)} \cup \sigma_{(k+1,0)}) \subset \bigcup_{B \in \mathcal{G}_k} B \subset (\sigma_{(k+1,0)} \cup
    \sigma_{(k,0)})$, thus
\begin{equation}\label{struc-spec}
\Sigma = \bigcap_{k \ge 0} \bigcup_{B \in \mathcal{G}_k} B;
\end{equation}
\item any $(k,I)$-type band contains only one band in $\mathcal{G}_{k+1}$, which is of $(k+1,II)$-type; \item any
    $(k,II)$-type band contains $2 a_{k+1} + 1$ bands in $\mathcal{G}_{k+1}$, $a_{k+1}+1$ of which are of $(k+1,I)$-type
    and $a_{k+1}$ of which are of $(k+1,III)$-type; \item any $(k,III)$-type band contains $2 a_{k+1} - 1$ bands in
    $\mathcal{G}_{k+1}$, $a_{k+1}$ of which are of $(k+1,I)$-type and $a_{k+1}-1$ of which are of $(k+1,III)$-type.
\end{itemize}
Thus $\{ \mathcal{G}_k \}_{k \ge 0}$ forms a natural covering of the spectrum $\Sigma$. %For any $k\ge1$, let $s_k$ be the
unique real number in $[0,1]$ satisfies %$$\sum_{B\in\mathcal{G}_k}|B|^{s_k}=1,$$ %and define the pre-dimensions of $\Sigma $
by %\begin{equation}\label{predim} %s_*(\lambda)=\liminf_{k\rightarrow\infty}s_k,\quad
%s^*(\lambda)=\limsup_{k\rightarrow\infty}s_k. %\end{equation}

In the following we will give a coding for $\Sigma$. Let
$$
\mathcal E := \{ (I,II), (II,I), (II,III), (III,I), (III,III) \}
$$
be the admissible edges. To simplify the notation, we write
$$
e_{12} = (I,II), e_{21} = (II,I), e_{23} = (II,III), e_{31} = (III,I), e_{33} = (III,III).
$$
For each $n \in\N$, define
$$
\tau_{e}(n)=
\begin{cases}
1& e=e_{12}\\ n+1& e=e_{21}\\ n& e=e_{23}\\ n& e=e_{31}\\ n-1& e=e_{33}.
\end{cases}
$$
Then define
\begin{eqnarray*}
\mathcal{E}_n&=&\{ (e,\tau_e(n),l)\ :\ e\in\mathcal E,\ 1\le l\le \tau_e(n)\}\\ \mathcal{E}_n^\ast&=&\{
(e,\tau_e(n),l)\in\mathcal{E}_n\ :\  e\ne e_{21},e_{23} \}.
\end{eqnarray*}

For any $w = (e, \tau_e(n), l) \in \mathcal{E}_n$, we use the notation $e_w := e.$%?

For any $n, n^\prime \in \N$ and any $(e, \tau_e(n), l) \in \mathcal{E}_n$ and $(e^\prime, \tau_e^\prime(n^\prime), l^\prime)
\in \mathcal{E}_{n^\prime}$, we say $(e, \tau_e(n), l) (e^\prime, \tau_e^\prime(n^\prime), l^\prime)$ is {\it admissible} if
the end point of $e$ is the initial point of $e^\prime$. We denote it by $(e, \tau_e(n), l) \to (e^\prime,
\tau_e^\prime(n^\prime), l^\prime)$.

Define
$$
\Omega = \{ \omega \in \mathcal{E}_{a_1}^\ast \times \prod_{m=2}^\infty \mathcal{E}_{a_m} : \omega = \omega_1 \omega_2 \cdots
\text{ s.t. } \omega_{m} \to \omega_{m+1} \text{ for all } m \ge 1 \}.
$$
Define $\Omega_1 = \mathcal{E}_{a_1}^\ast$ and for $k \ge 2,$ define
$$
\Omega_k = \{ w \in \mathcal{E}_{a_1}^\ast \times \prod_{m=2}^k \mathcal{E}_{a_m} : w = w_1 \cdots w_k \text{ s.t. } w_{m} \to
w_{m+1} \text{ for all } 1 \le m < k \}.
$$
Define finally $\Omega_\ast = \bigcup_{k \ge 1} \Omega_k$.

Given any $w \in \Omega_k$, $1 \le m < k$, we write $w = u*v$ or $w = uv$, where $u = w_1 \cdots w_m$, $v = w_{m+1} \cdots
w_k$.

Given any $w \in \Omega_k$, define $B_w$ inductively as follows: Let $B_{I} = [\lambda - 2, \lambda + 2]$ be the unique
$(0,I)$-type band in $\mathcal{G}_0$ and let $B_{III} = [-2,2]$ be the unique $(0,III)$-type band in $\mathcal{G}_0$.

Let $w \in \Omega_1$ be given. If $w = (e_{12},1,1)$, then define $B_w$ to be the unique $(1,II)$-type band contained in
$B_I$. If $w = (e_{31},\tau_{e_{31}}(a_1),l)$, then define $B_w$ to be the unique $l$-th $(1,I)$-type band contained in
$B_{III}.$ If $w = (e_{33},\tau_{e_{33}}(a_1),l)$, then define $B_w$ to be the unique $l$-th $(1,III)$-type band contained in
$B_{III},$ where we order the bands of the same type from left to right.

Suppose $B_w$ has been defined for any $w \in \Omega_{k-1}$. Given $w \in \Omega_k$ and write $w = w^\prime \ast
(e,\tau_e(a_k),l)$. Then $w^\prime \in \Omega_{k-1}$. If $e = (T,T^\prime)$, define $B_w$ to be the unique $l$-th
$(k,T^\prime)$-type band inside $B_{w^\prime}$.

With these notations we can rewrite \eqref{struc-spec} as
$$
\Sigma = \bigcap_{k \ge 0} \bigcup_{w \in \Omega_k} B_w.
$$

Given $w \in \Omega_k$, we say $w$ has length $k$ and denote by $|w| = k$. If $B_w$ is of $(k,T)$ type, sometimes we also say
simply that $B_w$ has type $T$. We will write $h_w$ for the generating polynomial of $B_w$. Moreover, given a band $B$, we
will also write $t_B$ for its generating polynomial, that is, if $B = B_w$, then $t_B = h_w$.

\section{The Length of the Longest Band at a Given Level}

The following lemma is \cite[Lemma~3.6]{LQW14} (see also \cite[Proposition 3.3]{FLW11}).

\begin{lemma}\label{quo-deri}
Assume $\lambda\ge 20.$ Assume $w \in \Omega_k, wu \in \Omega_{k+1}$ with $u = (e,p,l)$. Let $h_w, h_{wu}$ be the generating
polynomials of $B_w, B_{wu}$, respectively. Then for any $z \in B_{wu}$, if $e \ne e_{12}$,
\begin{equation}\label{eq-1}
\frac{\lambda-8}{3} (p+1) \csc^2 \frac{l \pi}{p+1} \le \left| \frac{h_{wu}^\prime(z)}{h_w^\prime(z)} \right| \le
(\lambda+5)(p+1) \csc^2 \frac{l \pi}{p+1},
\end{equation}
if $e = e_{12}$, then $p=1$, we have
\begin{equation}\label{eq-2}
\left( \frac{2(\lambda-8)}{3} \right)^{a_{k+1}-1} \le \left| \frac{h_{wu}^\prime(z)}{h_w^\prime(z)} \right| \le \left(
2(\lambda+5) \right)^{a_{k+1}-1}.
\end{equation}
\end{lemma}
We remark that here $p = \tau_e(a_{k+1})$.

The above lemma has the following consequence:

\begin{lemma}\label{lem-bc}
Assume $\lambda \ge 20$. Write $t_1 = (\lambda-8)/3$ and $t_2 = 3(\lambda+5)$. Then for any $w = w_1 \cdots w_k \in \Omega_k$
with $w_i = (e_i, \tau_{e_i}(a_i), l_i)$, we have
\begin{equation}\label{bd-basic}
\prod_{e_i = e_{12}} \frac{1}{t_2^{a_i-1}} \cdot \prod_{e_i \ne e_{12}} \frac{1}{t_2 a_i} \cdot \prod_{e_i \ne e_{12}} \sin^2
\frac{l_i \pi}{\tau_{e_i}(a_i) + 1} \le |B_w| \le 4 \prod_{e_i = e_{12}} \frac{1}{t_1^{a_i-1}} \cdot \prod_{e_i \ne e_{12}}
\frac{1}{t_1 a_i}.
\end{equation}
\end{lemma}

\begin{proof}
Given $w \in \Omega_k$. Consider the initial ladder $(B_i)_{i=0}^k$ with $B_0$ the unique band in $\mathcal{G}_0$ containing
$B_w$ and $B_k = B_w$. Let $(\hat{B}_i)_{i=0}^m$ be the related modified ladder (cf.~\cite{FLW11}) and $(\hat{h}_i)_{i=0}^m$,
$(p_i)_{i=0}^{m-1}$ and $(l_i)_{i=0}^{m-1}$ be the corresponding generating polynomials, type sequence and index sequence.
Since $\hat h_m(\hat B_m) = [-2,2]$, there exists $z_0 \in \hat B_m$ such that $|\hat h_m^\prime(z_0)| |\hat B_m| = 4$. Notice
also that $|\hat{h}_{0}^\prime| \equiv 1$ (see the explanation after Definition~\ref{def1}), then by Proposition~6.3 of
\cite{LQW14}, the definition of modified ladder and (44) of \cite{LQW14}
\begin{eqnarray*}
|B_w| & = & |\hat{B}_m| = 4 \frac{|\hat h_0^\prime(z_0)|}{|\hat h_m^\prime(z_0)|} \le 4 \prod_{i=0}^{m-1} \frac{3 \sin^2
\frac{l_i \pi}{\tau_{e_i}(a_i) + 1}}{(\lambda-8)(\tau_{e_i}(a_i) + 1)} \\ & \le & 4 \prod_{e_j = e_{12}}
\frac{1}{(2t_1)^{a_j-1}} \cdot \prod_{e_j \ne e_{12}} \frac{1}{(\tau_{e_j}(a_j)+1)t_1} \\ & \le & 4 \prod_{e_j = e_{12}}
\frac{1}{t_1^{a_j-1}} \cdot \prod_{e_j \ne e_{12}} \frac{1}{a_jt_1}.
\end{eqnarray*}
Similarly, by using the fact that $\tau_e(n) + 1 \le 3n$, we have
\begin{eqnarray*}
|B_w| & \ge & 4 \prod_{i=0}^{m-1} \frac{\sin^2 \frac{l_i \pi}{p_i+1}}{(\lambda+5)(\tau_{e_i}(a_i) + 1)} \\ & = & 4 \prod_{e_j
= e_{12}} \frac{1}{(2(\lambda+5))^{a_j-1}} \cdot \prod_{e_j \ne e_{12}} \frac{\sin^2 \frac{l_j \pi}{\tau_{e_j}(a_j) +
1}}{(\tau_{e_j}(a_j)+1)(\lambda+5)} \\ & \ge & \prod_{e_j = e_{12}} \frac{1}{t_2^{a_j-1}} \cdot \prod_{e_j \ne e_{12}}
\frac{1}{a_j t_2} \cdot \prod_{e_j \ne e_{12}} \sin^2 \frac{l_j \pi}{\tau_{e_j}(a_j) + 1}.
\end{eqnarray*}
\end{proof}

Write
$$
\delta_k = (a_1 \cdots a_k)^{1/k}.
$$
For any $w \in \Omega_k$, define
$$
|w|_\ast := \# \{ 1 \le i \le k : e_i = e_{12} \ \ \text{ and }\ \ a_i = 1\}.
$$
We have the following corollary.

\begin{coro}
Assume $\lambda \ge 20$. Write $t_1 = (\lambda-8)/3$ and $t_2 = 3(\lambda+5)$. Then for any $w = w_1 \cdots w_k \in \Omega_k$
with $w_i = (e_i,\tau_{e_i}(a_i),l_i)$, we have
\begin{equation}\label{bd-basic-1}
|B_w| \le 48^k \delta_k^{-k} \lambda^{|w|_{\ast}-k}.
\end{equation}
If moreover we take $l_i = 1$ when $a_i = 1,2$ and $l_i = \lfloor (\tau_{e_i}(a_i) + 1)/2 \rfloor$ when $a_i \ge 3$, then
\begin{equation}\label{bd-basic-2}
|B_w| \ge 8^{-k}\delta_k^{-k} \lambda^{-k} \prod_{e_i=e_{12}} \frac{a_i}{t_2^{a_i-2}}.
\end{equation}
\end{coro}

\begin{proof}
First, by \eqref{bd-basic} we have
$$
|B_w| \le 4 \prod_{e_i = e_{12}} \frac{1}{t_1^{a_i-1}} \cdot \prod_{e_i \ne e_{12}} \frac{1}{t_1 a_i} = 4 \prod_{e_i = e_{12}}
\frac{a_i}{t_1^{a_i-2}} \cdot \prod_{i=1}^n \frac{1}{t_1 a_i}.
$$
Notice that
$$
\frac{a_i}{t_1^{a_i-2}}\ \
\begin{cases}
=t_1& a_i=1\\ =2&a_i=2\\ \le 2&a_i\ge 3
\end{cases}.
$$
Since $\lambda \ge 20$, we have $\lambda/6 \le t_1 \le \lambda$. Consequently we get
$$
|B_w| \le 48^k \delta_k^{-k} \lambda^{|w|_{\ast}-k}.
$$

In the following we take $l_i = 1$ for $a_i = 1,2$ and $l_i = \lfloor (\tau_{e_i}(a_i) + 1)/2 \rfloor$ for $a_i \ge 3$. Then
it is easy to show that $\pi/4 \le l_i \pi/(\tau_{e_i}(a_i) + 1) \le \pi/2$. By \eqref{bd-basic},
$$
|B_w| \ge 2^{-k} \prod_{e_i = e_{12}} \frac{1}{t_2^{a_i-1}} \cdot \prod_{e_i \ne e_{12}} \frac{1}{t_2 a_i} = 2^{-k} \prod_{e_i
= e_{12}} \frac{a_i}{t_2^{a_i-2}} \cdot \prod_{i=1}^k \frac{1}{t_2 a_i}.
$$
Since $\lambda \ge 20$ we have $t_2 \le 4 \lambda$, thus we conclude that
$$
|B_w| \ge 8^{-k} \delta_k^{-k} \lambda^{-k} \prod_{e_i = e_{12}} \frac{a_i}{t_2^{a_i-2}}.
$$
\end{proof}

We have the following estimates for the maximal length of the bands of order $k$.

\begin{prop}\label{p.maxband}
Consider the integer sequence $a_1 \cdots a_k$. It may be divided by $1$'s into several segments that do not contain $1$'s,
that is, we can write
$$
a_1 \cdots a_k = A_1 1^{m_1} A_2 1^{m_2} A_3 \cdots A_s 1^{m_s} A_{s+1},
$$
where for any $i = 1, \cdots, s+1$, $A_i = a_m a_{m+1} \cdots a_{m+l}$ for some $m > 0$, $l \ge 0$, $a_j > 1$ for any $m \le j
\le m+l$, and $m_i \ge 1$. Note that if $a_1 = 1$, then $A_1 = \emptyset$. Assume $|B_{\hat w}| = \max \{ |B_w| : w \in
\Omega_n \}$, then
\begin{equation}\label{main-esti}
8^{-k} \cdot \delta_k^{-k} \cdot \lambda^{-k+\sum_{j=1}^s \lfloor (m_j+1)/2 \rfloor} \le |B_{\hat w}| \le 48^k \cdot
\delta_k^{-k} \cdot \lambda^{-k+\sum_{j=1}^s \lfloor (m_j+1)/2 \rfloor}.
\end{equation}
\end{prop}

\begin{proof}
Given any $w \in \Omega_n$, let us show that
\begin{equation}\label{length}
|w|_\ast \le \sum_{j=1}^s \left\lfloor \frac{m_j+1}{2} \right\rfloor.
\end{equation}
Indeed by the definition of admissibility we know that $e_{12} e_{12}$ is not admissible. Thus for each block $1^{m_j} =
a_{m+1} \cdots a_{m+m_j}$,
$$
\# \{ m+1 \le j \le m+m_j : e_j = e_{12} \} \le \left\lfloor \frac{m_j+1}{2} \right\rfloor.
$$
Consequently \eqref{length} holds.

Now by \eqref{bd-basic-1} we get
$$
|B_{w}| \le 48^k \cdot \delta_k^{-k} \cdot \lambda^{-k+\sum_{j=1}^s \lfloor (m_j+1)/2 \rfloor}.
$$
In particular, this inequality holds for $B_{\hat w}$. Thus we get the second inequality of \eqref{main-esti}.

Next we will construct a special $\tilde w \in \Omega_n$. Write  $|A_j| = n_j$ and $\tau_m = \sum_{j=1}^m (n_j+m_j).$ At first
we define $e_j, j = 1, \cdots, \tau_s+n_{s+1}$ by induction.

For $j = 1, \cdots, \tau_1$, we discuss four cases:

Case 1: $n_1=0$ and $m_1$ is odd. Define
$$
e_1 \cdots e_{m_1} = e_{12} (e_{21} e_{12})^{(m_1-1)/2}.
$$

Case 2: $n_1=0$ and $m_1$ is even. Define
$$
e_1 \cdots e_{m_1} = e_{31} e_{12} (e_{21} e_{12})^{m_1/2-1}.
$$

Case 3: $n_1>0$ and $m_1$ is odd. Define
$$
e_1 \cdots e_{n_1} \cdot e_{n_1+1} \cdots e_{n_1+m_1} = e_{33}^{n_1-1} e_{31} \cdot e_{12} (e_{21} e_{12})^{(m_1-1)/2}.
$$

Case 4: $n_1>0$ and $m_1$ is even. Define
$$
e_1 \cdots e_{n_1} \cdot e_{n_1+1} \cdots e_{n_1+m_1} = e_{33}^{n_1} \cdot e_{31} e_{12} (e_{21}e_{12})^{m_1/2-1}.
$$

Assume $e_j$ is already defined for $j \le \tau_{i-1}$. Now define $e_j$, $j = \tau_{i-1}+1, \cdots, \tau_{i}$ as follows.

Case 1: $n_i=1$ and $m_{i}$ is odd. Define
$$
e_{\tau_{i-1}+1} \cdots e_{\tau_{i-1}+n_{i}} \cdot e_{\tau_{i-1}+n_{i}+1} \cdots e_{\tau_{i}} = e_{21} \cdot e_{12}
(e_{21}e_{12})^{(m_1-1)/2}.
$$

Case 2: $n_i \ge 2$ and $m_{i}$ is odd. Define
$$
e_{\tau_{i-1}+1} \cdots e_{\tau_{i-1}+n_{i}} \cdot e_{\tau_{i-1}+n_{i}+1} \cdots e_{\tau_{i}} = e_{23} e_{33}^{n_i-2} e_{31}
\cdot e_{12} (e_{21} e_{12})^{(m_1-1)/2}.
$$

Case 3: $m_i$ is even. Define
$$
e_{\tau_{i-1}+1} \cdots e_{\tau_{i-1}+n_{i}} \cdot e_{\tau_{i-1}+n_{i}+1} \cdots e_{\tau_{i}} = e_{23} e_{33}^{n_i-1} \cdot
e_{31} e_{12} (e_{21} e_{12})^{m_i/2-1}.
$$
Thus by induction we have defined $e_j$ for $j = 1, \cdots, \tau_s$. Finally, if $n_{s+1}>0$, then define
$$
e_{\tau_s+1} \cdots e_{\tau+n_{s+1}} = e_{23} e_{33}^{n_{s+1}-1}.
$$
Now define $\tilde w_j = (e_j, \tau_{e_j}(a_j), l_j)$ such that
$$
l_j=
\begin{cases}
1& a_j=1,2\\ \lfloor (\tau_{e_j}(a_j)+1)/2 \rfloor & a_j \ge 3.
\end{cases}
$$
Then for $\tilde w = \tilde w_1 \cdots \tilde w_k$, by construction it is seen that
$$
|\tilde w|_\ast= \sum_{j=1}^s \left\lfloor \frac{m_j+1}{2} \right\rfloor.
$$
Moreover $e_j = e_{12}$ only if $a_j = 1$. Since $t_2 \ge \lambda$, by \eqref{bd-basic-2} we have
$$
|B_{\tilde w}| \ge 8^{-k} \delta_k^{-k} \lambda^{-k} \lambda^{\sum_{j=1}^s \lfloor (m_j+1)/2 \rfloor}.
$$
Since we have $|B_{\hat w}| \ge |B_{\tilde w}|$, the first inequality of \eqref{main-esti} holds.
\end{proof}

\begin{exam}
{\rm 1) If $a_j \ge 2$, then the maximal length is about
$$
(a_1 \cdots a_k)^{-1} \lambda^{-k}.
$$

2) If $a_j = 1$, then the maximal length is about
$$
\lambda^{-k/2}.
$$

3) If $a_{2j} = 1$ and $a_{2j+1} \ge 2$, then the maximal length is about
$$
(a_1 \cdots a_k)^{-1} \lambda^{-k/2}.
$$

4) If $a_1 a_2 \cdots = 112112112 \cdots $, then the maximal length is about
$$
(a_1 \cdots a_k)^{-1} \lambda^{-2k/3}\simeq 2^{-k/3}\lambda^{-2k/3}.
$$

5) If $a_1 a_2 \cdots = 111212111212111212\cdots $, then the maximal length is about
$$
(a_1 \cdots a_k)^{-1} \lambda^{-k/2}\simeq 2^{-k/3}\lambda^{-k/2}.
$$

}
\end{exam}

\begin{remark}
Notice that the examples above show that the asymptotics of the the length of the longest band is not a function of the
frequencies of the continued fraction coefficients in general.
\end{remark}

\section{The Connection Between the Longest Spectral Generating Bands and the Upper Transport Exponent}

For any $k \ge -1$, let us write $x_k(z) = \tr M_k(z)$ and $y_k(z) = \tr [M_{k-1}(z) M_k(z)]$. Notice that $x_k(z) =
t_{(k+1,0)}(z)$ and $y_k(z) = t_{(k,1)}(z)$. The following lemma was stated as \cite[Lemma~2]{M10}.

\begin{lemma}\label{l.escape}
Suppose $\delta \ge 0$ and $z \in \C$. A necessary and sufficient condition for $\{ x_k(z) \}_{k \ge -1}$ to be unbounded is
that there exists $k_0 \ge 0$ such that
\begin{equation}\label{e.escape}
|x_{k_0 - 1}(z)| \le 2 + \delta, \quad |x_{k_0}(z)| > 2 + \delta, \quad |y_{k_0}(z)| > 2 + \delta.
\end{equation}
In this case, this $k_0$ is unique, and with
$$
G_k = G_{k-1} + a_k G_{k-2} , \quad G_0 = G_{-1} = 1,
$$
we have
$$
|x_{k+1}(z)| \ge |y_k(z)| \ge (1 + \delta)^{G_{k - k_0}} + 1, \quad \text{ for every } k > k_0.
$$
\end{lemma}

\begin{remark}
Since this lemma is important in what follows, we need to comment on it and its proof. While it was shown that the condition
is sufficient in \cite{M10}, it wasn't shown there that it is necessary. This part of the proof is simply missing in
\cite{M10}. Moreover, the following example shows that the condition is in fact not necessary.

Consider $\lambda = 3$, any frequency $\alpha$ with $a_1 = 3$ and $a_2 = 1$, and $z = 3$. Then, one computes that
$$
x_{-1}(z) = 2, \; x_0(z) = 3, \; y_0(z) = 0, \; x_1(z) = 6, \; y_1(z) = -16, \; x_2(z) = -16.
$$
Since $z$ is real, \cite[Proposition~4]{BIST89} applies and yields that starting at $k = 0$, the super-exponential escape
kicks in, no matter how the subsequent continued fraction coefficients are chosen. In particular, we have $|x_k(z)| > 2$ for
every $k \ge 0$. This shows that, while $\{ x_k(z) \}_{k \ge -1}$ is unbounded, there exists no $k_0 \ge 0$ such that
\eqref{e.escape} holds.
\end{remark}

We modify Lemma \ref{l.escape} as follows.

\begin{lemma}\label{modifymarin}
Let $\delta \ge 0$ and $z \in \C$. Suppose that there exists $k_0 \ge 0$ such that
\begin{equation}\label{e.escape.new}
|x_{k_0}(z)| > 2 + \delta, \quad |y_{k_0}(z)| > 2 + \delta, \quad |y_{k_0}(z)| > |x_{k_0-1}(z)|.
\end{equation}
Let
\begin{equation}\label{e.MarinGkdef}
G^{(k_0)}_{0}=1,\quad G^{(k_0)}_{1} = a_{k_0+1}, \quad G^{(k_0)}_{k+1} = a_{k_0 + k + 1} G^{(k_0)}_{k} +  G^{(k_0)}_{k-1} , \;
k \ge 1,
\end{equation}
Then, we have
\begin{equation}\label{cont_frac}
|x_{k_0+k}(z)| > (1 + \delta)^{G^{(k_0)}_{k }} + 1, \quad \text{ for every } k \ge 0.
\end{equation}
\end{lemma}

\begin{proof}
By \eqref{e.escape.new}, we have \eqref{cont_frac} for $k=0$. Instead of proving \eqref{cont_frac} for $k \ge 1$, we prove by
induction the following inequalities for any $k\ge1$,
\begin{equation}\label{induction}
\begin{array}{l}
|x_{k_0+k}(z)| > (1+\delta)^{G^{(k_0)}_{k}}+1\\ |y_{k_0+k}(z)| > (|x_{k_0+k-1}(z)|-1)|x_{k_0+k}(z)|\\ |y_{k_0+k}(z)| >
|x_{k_0+k-1}(z)|.
\end{array}
\end{equation}
Note that since $|x_{k_0}(z)| > 2+\delta$ and $a,b>2$ implies $(a-1)b>a$, the second inequality always implies the third
inequality. We write them in this form in order to apply the following claim.

\noindent {\bf Claim.} For any $k\ge0$, $p\ge1$, if %\begin{equation}\label{iter-condition}
$$
|t_{(k+1,0)}(z)| > 2+\delta, \quad |t_{(k,p)}(z)| > 2+\delta, \quad |t_{(k,p)}(z)| > |t_{(k,p-1)}|,
$$%\end{equation}
then %\begin{equation}\label{iteration}
$$|t_{(k,p+1)}(z)|>(|t_{(k+1,0)}(z)|-1)|t_{(k,p)}(z)|.
$$%\end{equation}

The claim is a direct result of the recursion \eqref{e.raymondrec}.

By the fact that for any $k\ge0$,
\begin{equation}\label{id}
t_{(k,a_{k+1})}(z) = t_{(k+2,0)}(z) = x_{k+1}(z), \quad t_{(k,a_{k+1}+1)}(z) = t_{(k+1,1)}(z) = y_{k+1}(z),
\end{equation}
the condition \eqref{e.escape.new} is equivalent to
$$
|t_{(k_0+1,0)}(z)| > 2+\delta, \quad |t_{(k_0,1)}(z)| > 2+\delta, \quad |t_{(k_0,1)}(z)| > |t_{(k_0,0)}(z)|.
$$
By our claim and induction, we have
\begin{equation}\label{initial}
\begin{array}{l}
|x_{k_0+1}(z)| > (|x_{k_0}(z)|-1)^{a_{k_0+1}-1} |y_{k_0}(z)| > (1+\delta)^{a_{k_0+1}}+1 \\ |y_{k_0+1}(z)| >
(|x_{k_0}(z)|-1)|x_{k_0+1}(z)| \\ |y_{k_0+1}(z)| > |x_{k_0}(z)|.
\end{array}
\end{equation}
This implies \eqref{induction} for $k=1$.

Suppose \eqref{induction} hold for any $n$ with $1\le n\le k$ for some $k\ge1$. Then
$$
\begin{array}{rcl}
|x_{k_0+k+1}(z)| & = & |t_{(k_0+k, a_{k_0+k+1})}(z)| \\ & > & (|x_{k_0+k}(z)|-1)^{a_{k_0+k+1}-1} |y_{k_0+k}(z)| \\ & > &
(|x_{k_0+k}(z)|-1)^{a_{k_0+k+1}-1}(|x_{k_0+k-1}(z)|-1)|x_{k_0+k}(z)| \\ & > &
(|x_{k_0+k}(z)|-1)^{a_{k_0+k+1}}(|x_{k_0+k-1}(z)|-1)+1 \\ & > & (1+\delta)^{a_{k_0+k+1} G^{(k_0)}_{k}+G^{(k_0)}_{k-1}}+1 \\ &
= & (1+\delta)^{G^{(k_0)}_{k+1}}+1,
\end{array}
$$
where the first inequality is due to the Claim, \eqref{id} and induction, the second and the third inequalities are by the
induction hypothesis, for $k > 1$, the forth inequality is because of the induction hypothesis, for $k=1$, the forth
inequality is because of $|x_{k_0}(z)| > 2+\delta$.

And also by the claim and \eqref{id}, we have
$$
|y_{k_0+k+1}(z)| > (|x_{k_0+k}(z)|-1)|x_{k_0+k+1}(z)|, \quad |y_{k_0+k+1}(z)| > |x_{k_0+k}(z)|.
$$

By induction, we get \eqref{induction} and hence \eqref{cont_frac} for all $k\ge1$.
\end{proof}

Now we can give a necessary and sufficient criterion for $\{ x_k(z) \}_{k \ge -1}$ to be unbounded.

\begin{lemma}\label{l.complexenergyBISTlemma}
Let $\delta \ge 0$.

{\rm (a)} Given $z \in \C$, a necessary and sufficient condition for $\{ x_k(z) \}_{k \ge -1}$ to be unbounded is that there
exists $k_0 \ge 0$ such that
\begin{equation}\label{e.escape.criteria}
|x_{k_0 - 1}(z)| \le 2 + \delta, \quad |x_{k_0}(z)| > 2 + \delta, \quad |x_{k_0+1}(z)| > 2 + \delta.
\end{equation}

{\rm (b)} If there is a $k_0$ such that \eqref{e.escape.criteria} holds, then \eqref{e.escape.criteria} holds for no other
value of $k_0$ and we have $|x_k(z)| \ge (1 + \delta)^{G^{(k_0)}_{k - k_0}} + 1$, with $G$ from \eqref{e.MarinGkdef}, for $k
\ge k_0$.

{\rm (c)} There exists a constant $C_{\lambda,\delta} \in (0,\infty)$ such that for every $z \in \C$, the following holds. If
$|x_{k_1}(z)| \le 2 + \delta$, then we have
\begin{equation}\label{e.boundforxk}
|x_k(z)| \le C_{\lambda,\delta}
\end{equation}
for every $k \in \{ 0, \ldots, k_1 \}$. In particular, if $\{ x_k(z) \}_{k \ge -1}$ is bounded, then we have
\eqref{e.boundforxk} for every $k \ge -1$.
\end{lemma}

\begin{remark}
This lemma is a version of \cite[Proposition~4]{BIST89} for complex energies. Note that \cite[Proposition~4]{BIST89} is a
statement for real energies and $\delta = 0$. Indeed, its proof in \cite{BIST89} makes crucial use of the fact that the energy
$z$ in question is real. Given the method we use to estimate transport exponents, we absolutely do need a version of this
statement for energies that are not real. Our proof of Lemma~\ref{l.complexenergyBISTlemma} not only works for general complex
energies $z$ (and $\delta \ge 0$), it is also simpler than the proof of \cite[Proposition~4]{BIST89}.

As pointed out in \cite{BIST89}, a statement like \cite[Proposition~4]{BIST89} or Lemma~\ref{l.complexenergyBISTlemma} is of
fundamental importance in an analysis of Sturmian Hamiltonians that is based on the trace map approach. In particular, such a
result is necessary in many extensions of results from the Fibonacci case to the Sturmian case, and the subsequent discussion
in this section is yet another instance of this.
\end{remark}

\begin{proof}[Proof of Lemma~\ref{l.complexenergyBISTlemma}.]
(a) We first show that the condition is sufficient. Suppose that $k_0$ satisfies \eqref{e.escape.criteria}. Then,
equivalently,
$$
|t_{(k_0,0)}(z)| \le 2 + \delta, \quad |t_{(k_0+1,0)}(z)| > 2 + \delta, \quad |t_{(k_0,a_{k_0+1})}(z)| = |t_{(k_0+2,0)}(z)| >
2 + \delta.
$$
So in the sequence $(t_{(k_0,p)}(z))_{p=0}^{a_{k_0+1}}$, there exists $0 < p \le a_{k_0+1}$ such that
$$
|t_{(k_0,p-1)}(z)| \le 2 + \delta < |t_{(k_0,p)}(z)|.
$$
By the claim in Lemma \ref{modifymarin} and induction, we find
$$
\begin{array}{rcl}
|t_{(k_0,a_{k_0+1})}(z)| & \ge & ( t_{(k_0+1,0}(z)|-1 )^{a_{k_0+1}-p} |t_{(k_0,p)}(z)| \\ |t_{(k_0,a_{k_0+1}+1)}(z)| & > & (
t_{(k_0+1,0}(z)|-1 ) |t_{(k_0,a_{k_0+1})}(z)|.
\end{array}
$$
Hence
$$
|y_{k_0+1}(z)| = |t_{(k_0,a_{k_0+1}+1)}(z)| > 2 + \delta, \quad |y_{k_0+1}(z)| > |x_{k_0}(z)|.
$$
Since we already have $|x_{k_0+1}(z)| > 2 + \delta$, by Lemma \ref{modifymarin}, $\{ x_k(z) \}_{k \ge -1}$ is unbounded.

\medskip

We now show that the condition is necessary. Since $x_{-1}(z) = 2$, the non-existence of $k_0$ such that
\eqref{e.escape.criteria} holds is equivalent to the non-existence of $k_0$ such that
$$
|x_{k_0}(z)| > 2+\delta, \quad |x_{k_0+1}(z)| > 2+\delta.
$$

Suppose there is no $k_0$ such that \eqref{e.escape.criteria} holds. Fix $k > 0$ and assume $|x_k(z)| > 2 + \delta$. Then,
$$
|t_{(k,0)}(z)| = |x_{k-1}(z)| \le 2 + \delta.
$$
Moreover,
$$
|t_{(k,1)}(z)| = |y_k(z)| \le 2 + \delta,
$$
for otherwise $|y_k(z)| > 2+\delta$, together with $|x_k(z)| > 2+\delta$ and $|x_{k-1}(z)| \le 2+\delta$, then by
\eqref{cont_frac},
$$
|t_{(k,a_{k+1})}(z)| = |x_{k+1}(z)| > 2 + \delta,
$$
a contradiction again.

Since
$$
\Lambda (x_{k-1}(z),x_{k}(z),y_k(z)) = \lambda^2,
$$
the estimates above imply that $|x_k(z)| \le C_{\lambda,\delta}$ for a suitable constant $C_{\lambda,\delta} \in (0,\infty)$.

(b) This statement follows from the argument used in proving the sufficiency of the condition in part (a) and
\eqref{cont_frac}.

(c) This statement follows from the argument used in proving the necessity of the condition in part (a).
\end{proof}

For $\delta \ge 0$, set
$$
\sigma_k^\delta = \{ z \in \C: |x_k(z)| \le 2 + \delta \}.
$$
Similarly, we also consider a complexification of $\mathcal{G}_k$ by replacing $2$ with $2 + \delta$ and the real preimage
with the complex preimage. Explicitly, for each $B \in \mathcal{G}_k$, which is a connected component of $\{ z \in \R :
|t_B(z)| \le 2 \}$ (i.e., $t_B$ is the generating polynomial of $B$, which is either $x_k$ or $y_k$), we pass to the set
$B^\delta$ which is given by the connected component of $\{ z \in \C : |t_B(z)| \le 2 + \delta \}$ that contains $B$. Then, we
consider
$$
\bigcup_{B \in \mathcal{G}_k} B^\delta \subset \C.
$$

\begin{lemma}\label{l.components}
For every $\lambda > 4$, there exists $\delta(\lambda) > 0$ such that for every $\delta \in [0,\delta(\lambda))$, every $k \ge
0$, we have the following statements.

{\rm (a)} For every $B \in \mathcal{G}_k$, $B^\delta$ contains a unique zero $z_B$ of its generating polynomial $t_B$.

{\rm (b)} We have
\begin{equation}\label{e.sigmakgk}
\sigma_k^\delta \cup \sigma_{k+1}^\delta \subseteq \bigcup_{B \in \mathcal{G}_k} B^\delta \subseteq \sigma_{k-1}^\delta \cup
\sigma_{k}^\delta.
\end{equation}
\end{lemma}

\begin{proof}
The choice of $\delta(\lambda)$ needs to ensure that we still have $\sigma^{\delta}_{(k,p+1)} \cap \sigma^{\delta}_{(k+1,0)}
\cap \sigma^{\delta}_{(k,p)} = \emptyset$ for all $\delta \in [0,\delta(\lambda))$, $k \ge 0$, and $p \ge -1$. By $\lambda >
4$ and the invariance condition \eqref{e.invariancecond}, such a choice is clearly possible.

(a) It is known that $t_B$ has a unique zero on each of its real components. This persists when $2$ is replaced by $2 +
\delta$ due to $\delta \in [0,\delta(\lambda))$ and our choice of $\delta(\lambda)$. The maximum modulus principle then
extends the statement to the complex component; compare the proof of \cite[Lemma~6]{DT07}.

(b) Recall that $\sigma_k \cup \sigma_{k+1} \subseteq \bigcup_{B \in \mathcal{G}_k} B \subseteq \sigma_{k-1} \cup \sigma_{k}$,
so we need a complex version of this statement. The complex version of the second inclusion holds by construction. We prove
the complex version of the first inclusion by mimicking the proof of the inclusion in the real case.

To do this we need a formula that is analogous to \eqref{tao}, that is,
\begin{equation}\label{tao.complex}
\sigma^{\delta}_{(k,p+1)} \subset \sigma^{\delta}_{(k+1,0)} \cup \sigma^{\delta}_{(k,p)}.
\end{equation}
Suppose this fails. Then, there is $z$ such that $|t_{(k,p+1)}(z)| \le 2 + \delta$, $|t_{(k+1,0)}(z)| > 2 + \delta$ and
$|t_{(k,p)}(z)| > 2 + \delta$. The recursion
$$
t_{k,p-1}(z) = t_{(k+1,0)}(z) t_{k,p}(z) - t_{(k,p+1)}(z)
$$
then implies $|t_{(k,p-1)}(z)| > 2 + \delta$. By an argument analogous to the one used in the proof of
Lemma~\ref{l.complexenergyBISTlemma}, it then follows that $|t_{(k,p+1)}(z)| > 2 + \delta$, which is a contradiction.

By definition of $\mathcal{G}_k$, it contains all components in $\sigma_{(k+1,0)}$. So we only need to consider components of
$\sigma_{(k+2,0)}$. Suppose $B_{(k+2,0)}$ is a component of $\sigma_{(k+2,0)}$ which is not contained in $\sigma_{(k+1,0)}$.
By \eqref{tao} and $\sigma_{(k,p+1)} \cap \sigma_{(k+1,0)} \cap \sigma_{(k,p)} = \emptyset$, we obtain
$$
B_{(k+2,0)} = B_{(k,a_{k+1})} \subset B_{(k,a_{k+1}-1)} \subset \cdots \subset B_{(k,1)}.
$$
So by \eqref{tao.complex} and $\sigma^{\delta}_{(k,p+1)} \cap \sigma^{\delta}_{(k+1,0)} \cap \sigma^{\delta}_{(k,p)} =
\emptyset$, it follows that
$$
B^{\delta}_{(k+2,0)} = B^{\delta}_{(k,a_{k+1})} \subset B^{\delta}_{(k,a_{k+1}-1)} \subset \cdots \subset B^{\delta}_{(k,1)}
\subset B^{\delta}_{(k,0)},
$$
where the last inclusion is due to
$$
\sigma^{\delta}_{(k+2,0)} \cup \sigma^{\delta}_{(k+1,0)} \subset \sigma^{\delta}_{(k+1,0)} \cup \sigma^{\delta}_{(k,0)},
$$
which is a consequence of Lemma~\ref{l.complexenergyBISTlemma}. But this shows that $B^{\delta}_{(k+2,0)}$ is contained in
$B^{\delta}_{(k,1)}$, which in turn is of the form $B^\delta$ for some $B \in \mathcal{G}_k$.
\end{proof}

%\begin{remark} %Clearly, the zero $z_B$ does not depend on $\delta \in [0,\delta(\lambda))$. %\end{remark}

\begin{lemma}\label{l.powerlawbound}
For every $\lambda > 4$, every $\delta \in [0,\delta(\lambda))$, and every bounded density number $\alpha$, there are
constants $C,\xi$ such that for every $k$, every $B \in \mathcal{G}_k$, every $z \in B^\delta$, and every $\omega \in \T$, we
have $\| M(n;\lambda,\alpha,\omega,z) \| \le C |n|^\xi$ for $1 \le |n| \le q_k$.
\end{lemma}

\begin{proof}
By Lemma~\ref{l.complexenergyBISTlemma} and \eqref{e.sigmakgk}, we have a uniform bound for $|x_{k'}|$, $k' = 0, \ldots, k$.
Using this observation, one can now mimic the proof of the power-law upper bound from \cite{IRT92} (phase zero) and
\cite{DL99} (general phase) for the energy $z$ and the sites $n$ in question.
\end{proof}

Let us define
\begin{align*}
r(B,\delta) & = \sup \{ r > 0 : B(z_B , r) \subseteq B^\delta \}, \quad r_k(\delta) = \max_{B \in \mathcal{G}_k} r(B,\delta),
\\ R(B,\delta) & = \inf \{ R > 0 : B(z_B , R) \supseteq B^\delta \}, \quad R_k(\delta) = \max_{B \in \mathcal{G}_k}
R(B,\delta).
\end{align*}

\begin{prop}\label{p.transportbounds}
Suppose $\lambda > 4$ and $\delta \in (0,\delta(\lambda)/2)$.

{\rm (a)} We have
\begin{equation}\label{e.rkRkxprimeconnection1}
\frac{1}{R_k(\delta)} \ge \frac{\delta^2}{(2 + \delta)(2 + 2\delta)^2} \min \{ |t_B'(z_B)| : B \in \mathcal{G}_k \}
\end{equation}
and
\begin{equation}\label{e.rkRkxprimeconnection2}
\frac{1}{r_k(\delta)} \le \frac{(4 + 3\delta)^2}{(2 + \delta)(2 + 2 \delta)^2} \min \{ |t_B'(z_B)| : B \in \mathcal{G}_k \}
\end{equation}
for every $k \ge 0$.

{\rm (b)} If $\alpha$ is such that $\lim_{k \to \infty} \frac1k \log q_k$ exists and is finite, then
\begin{equation}\label{e.alphauupperbound}
\tilde \alpha_u^+ \le \frac{\lim_{k \to \infty} \frac1k \log q_k}{\liminf_{k \to \infty} \frac{1}{k} \log
\frac{1}{R_k(\delta)}}
\end{equation}
If $\alpha$ is a bounded density number, then
\begin{equation}\label{e.alphaupluslowerbound}
\tilde \alpha_u^+ \ge \frac{\liminf_{k \to \infty} \frac1k \log q_k}{\limsup_{k \to \infty} \frac1k \log
\frac{1}{r_k(\delta)}}.
\end{equation}
If $\alpha$ is of bounded type, then
\begin{equation}\label{e.alphaulowerbound}
\tilde \alpha_u^- \ge \frac{\liminf_{k \to \infty} \frac1k \log q_k}{\limsup_{k \to \infty} \frac1k \log
\frac{1}{r_k(\delta)}}.
\end{equation}

{\rm (c)} If $\alpha$ is such that $\lim_{k \to \infty} \frac1k \log q_k$ exists and is finite, then
\begin{equation}\label{e.upperbound}
\tilde \alpha_u^+ \le \frac{\lim_{k \to \infty} \frac1k \log q_k}{\liminf_{k \to \infty} \frac{1}{k} \log \min \{ |t_B'(z_B)|
: B \in \mathcal{G}_k \}}.
\end{equation}
If $\alpha$ is a bounded density number, then
$$%\begin{equation}\label{e.alphaupluslowerbound}
\tilde \alpha_u^+ \ge \frac{\liminf_{k \to \infty} \frac1k \log q_k}{\limsup_{k \to \infty} \frac{1}{k} \log \min \{
|t_B'(z_B)| : B \in \mathcal{G}_k \}}.
$$%\end{equation}
If $\alpha$ is of bounded type, then
$$
\tilde \alpha_u^- \ge \frac{\liminf_{k \to \infty} \frac1k \log q_k}{\limsup_{k \to \infty} \frac{1}{k} \log \min \{
|t_B'(z_B)| : B \in \mathcal{G}_k \}}.
$$

In particular, suppose that $\alpha$ is of bounded type and the limits
$$
\lim_{k \to \infty} \frac1k \log q_k \quad \text{and} \quad \lim_{k \to \infty} \frac{1}{k} \log \min \{ |t_B'(z_B)| : B \in
\mathcal{G}_k \}
$$
exist. Then,
\begin{equation}\label{e.exponentsareequal}
\tilde \alpha_u^+ = \tilde \alpha_u^- = \frac{\lim_{k \to \infty} \frac1k \log q_k}{\lim_{k \to \infty} \frac{1}{k} \log \min
\{ |t_B'(z_B)| : B \in \mathcal{G}_k \}}.
\end{equation}

All statements in {\rm (b)} and {\rm (c)} above are uniform in the phase $\omega \in [0,1)$.
\end{prop}

\begin{remark}
By replacing the use of \cite{DT07} in the proof with the method of \cite{DT08}, one can obtain the same upper transport
bounds for the non-time-averaged upper transport exponent $\alpha_u^+$. In particular, due to the trivial inequality $\tilde
\alpha_u^+ \le \alpha_u^+$, one gets in \eqref{e.exponentsareequal} the chain of identities
$$
\alpha_u^+ = \tilde \alpha_u^+ = \tilde \alpha_u^- = \frac{\lim_{k \to \infty} \frac1k \log q_k}{\lim_{k \to \infty}
\frac{1}{k} \log \min \{ |t_B'(z_B)| : B \in \mathcal{G}_k \}},
$$
provided that the assumptions leading to \eqref{e.exponentsareequal} hold. % %(b) The only place where the bounded density
assumption enters the proof is when power-law upper bounds for transfer matrices are needed. This is completely analogous to
the result and proof in \cite{DKL00}. All other components of the proof work for much more general frequencies. This renews
the need for one to revisit the proof of power-law upper bounds for transfer matrices (cf.~\cite{DL99, IRT92}) in order to
determine whether it extends to a larger class of frequencies.
\end{remark}

\begin{proof}[Proof of Proposition~\ref{p.transportbounds}.]
(a) Let $\lambda > 4$ and choose $\delta \in (0,\delta(\lambda)/2)$. Fix $k$ and $B \in \mathcal{G}_k$, and consider
$B^{2\delta}$. Since $B^{2\delta}$ contains exactly one zero of $t_B$, it follows from the maximum modulus principle and
Rouch\'e's Theorem that
$$
t_B : \mathrm{int}(B^{2\delta}) \to B(0, 2 + 2\delta)
$$
is univalent, and hence
$$
t_B^{-1} : B(0, 2 + 2\delta) \to \mathrm{int}(B^{2\delta})
$$
is well-defined and univalent as well. Consequently, the following mapping is a Schlicht function:
$$
F : B(0,1) \to \C, \quad F(z) = \frac{t_B^{-1} ((2 + 2\delta)z) - z_B}{(2 + 2\delta) [(t_B^{-1})'(0)]}.
$$
That is, $F$ is a univalent function on $B(0,1)$ with $F(0) = 0$ and $F'(0) = 1$.

The Koebe Distortion Theorem (see \cite[Theorem~7.9]{C95}) implies that
\begin{equation}\label{koebe}
\frac{|z|}{(1 + |z|)^2} \le |F(z)| \le \frac{|z|}{(1 - |z|)^2} \text{ for } |z| < 1.
\end{equation}
Evaluate the bound \eqref{koebe} on the circle $|z| = \frac{2 + \delta}{2 + 2\delta}$. For such $z$, we obtain
$$
\frac{(2 + \delta)(2 + 2\delta)}{(4 + 3 \delta)^2} \le |F(z)| \le \frac{(2 + \delta)(2 + 2\delta)}{\delta^2}.
$$
By definition of $F$ this means that
$$
| t_B^{-1} ((2 + 2\delta)z) - z_B | \le \frac{(2 + \delta)(2 + 2\delta)}{\delta^2} (2 + 2\delta) |(t_B^{-1})'(0)|
$$
and
$$
| t_B^{-1} ((2 + 2\delta)z) - z_B | \ge \frac{(2 + \delta)(2 + 2\delta)}{(4 + 3 \delta)^2} (2 + 2\delta) |(t_B^{-1})'(0)|
$$
for all $z$ with $|z| = \frac{2 + \delta}{2 + 2\delta}$. In other words, if $|z| = 2 + \delta$, then
\begin{equation}\label{variation}
| t_B^{-1} (z) - z_B | \le \frac{(2 + \delta)(2 + 2\delta)^2}{\delta^2} |(t_B^{-1})'(0)|
\end{equation}
and
\begin{equation}\label{variation2}
| t_B^{-1} (z) - z_B | \ge \frac{(2 + \delta)(2 + 2\delta)^2}{(4 + 3\delta)^2} |(t_B^{-1})'(0)|.
\end{equation}
Note that as $z$ runs through the circle of radius $2 + \delta$ around zero, the point $t_B^{-1} (z)$ runs through the entire
boundary of $B^{\delta}$. Thus, since $|(t_B^{-1})'(0)| = |t_B'(z_B)|^{-1}$, \eqref{variation} and \eqref{variation2} yield
$$
B \Big( z_B, \frac{(2 + \delta)(2 + 2 \delta)^2}{(4 + 3\delta)^2} |t_B'(z_B)|^{-1} \Big) \subseteq B^{\delta} \subseteq B
\Big( z_B, \frac{(2 + \delta)(2 + 2\delta)^2}{\delta^2} |t_B'(z_B)|^{-1} \Big).
$$
In particular, it follows that
$$
\frac{(2 + \delta)(2 + 2 \delta)^2}{(4 + 3\delta)^2} |t_B'(z_B)|^{-1} \le r(B,\delta) \le R(B,\delta) \le \frac{(2 + \delta)(2
+ 2\delta)^2}{\delta^2} |t_B'(z_B)|^{-1}.
$$
Thus,
$$
\frac{\delta^2}{(2 + \delta)(2 + 2\delta)^2} |t_B'(z_B)| \le \frac{1}{R(B,\delta)} \le \frac{1}{r(B,\delta)} \le \frac{(4 +
3\delta)^2}{(2 + \delta)(2 + 2 \delta)^2} |t_B'(z_B)|,
$$
which in turn implies
$$
\frac{\delta^2}{(2 + \delta)(2 + 2\delta)^2} \left( \min \{ |t_B'(z_B)| : B \in \mathcal{G}_k \} \right) \le
\frac{1}{R_k(\delta)}
$$
and
$$
\frac{1}{r_k(\delta)} \le \frac{(4 + 3\delta)^2}{(2 + \delta)(2 + 2 \delta)^2} \left( \min \{ |t_B'(z_B)| : B \in
\mathcal{G}_k \} \right).
$$
This shows \eqref{e.rkRkxprimeconnection1}--\eqref{e.rkRkxprimeconnection2}.

(b) The Parseval identity implies (see, e.g., \cite[Lemma~3.2]{KKL03})
\begin{equation}\label{e.parsform}
2\pi \int_0^{\infty} e^{-2t/T} | \langle \delta_n , e^{-itH} \delta_0 \rangle |^2 \, dt = \int_{-\infty}^\infty \left|\langle
\delta_n  , (H - E - \tfrac{i}{T})^{-1} \delta_0 \rangle \right|^2 \, dE,
\end{equation}
and hence for the time averaged outside probabilities, defined by
\begin{equation}\label{e.taop}
\langle P(N,\cdot) \rangle (T) = \frac{2}{T} \int_0^{\infty} e^{-2t/T} \sum_{|n| \ge N} | \langle \delta_n , e^{-itH} \delta_0
\rangle |^2 \, dt,
\end{equation}
we have
\begin{equation}\label{e.taopresform}
\langle P(N,\cdot) \rangle (T) = \frac{1}{\pi T} \sum_{|n| \ge N} \int_{-\infty}^\infty \left|\langle \delta_n  , (H - E -
\tfrac{i}{T})^{-1} \delta_0 \rangle \right|^2 \, dE.
\end{equation}
The right-hand side of \eqref{e.taopresform} may be studied by means of transfer matrices at complex energies, which are
defined as follows. For $z \in \C$, $n \in \Z$, we set
$$
M(n;\lambda,\alpha,\omega,z) = \begin{cases} T(n;\lambda,\alpha,\omega,z) \cdots T(1;\lambda,\alpha,\omega,z) & n \ge 1, \\
T(n;\lambda,\alpha,\omega,z)^{-1} \cdots T(-1;\lambda,\alpha,\omega,z)^{-1} & n \le -1, \end{cases}
$$
where
$$
T(\ell;\lambda,\alpha,\omega,z) = \begin{pmatrix} z - \lambda \chi_{[1-\alpha,1)}(\ell \alpha + \omega \!\!\!\! \mod 1) & -1
\\ 1 & 0 \end{pmatrix}.
$$

By definition of $R_k(\delta)$, we have
$$
B^\delta \subseteq \{ z \in \C : |\mathrm{Im} \ z| \le R_k(\delta) \}
$$
for every $B \in \mathcal{G}_k$. If $\liminf_{k \to \infty} \frac1k \log \frac{1}{R_k(\delta)} = 0$, then
\eqref{e.alphauupperbound} holds trivially. Thus let us consider the case where $\liminf_{k \to \infty} \frac1k \log
\frac{1}{R_k(\delta)} > 0$. Then, for $\rho' > 0$ small enough, we have
\begin{equation}\label{e.sprimedef}
s' = \frac{\liminf_{k \to \infty} \frac1k \log \frac{1}{R_k(\delta)}}{\limsup_{k \to \infty} \frac1k \log q_k} - \rho' > 0.
\end{equation}
(Should $\liminf_{k \to \infty} \frac1k \log \frac{1}{R_k(\delta)}$ be infinite, we can work with $s'$ arbitrarily large.)
From the definition of $s'$ it follows that for some suitable $C_\delta' > 0$, we have
$$
R_k(\delta) < C_\delta' q_k^{- s'},
$$
for every $k \ge 0$. In particular, by \eqref{e.sigmakgk} in Lemma~\ref{l.components}, we have
\begin{equation}\label{imwidth}
\sigma_k^\delta \cup \sigma_{k+1}^\delta \subseteq \{ z \in \C : |\mathrm{Im} \ z| < C_\delta' q_k^{- s'} \}.
\end{equation}

For each $\varepsilon = \mathrm{Im} \ z > 0$, one obtains lower bounds on $|x_k(E+i\varepsilon)|$ which are uniform for $E \in
[-K,K] \subseteq \R$. Namely, given $\varepsilon > 0$, choose $k$ minimal with the property $C_\delta' q_{k-1}^{- s'} <
\varepsilon$. By \eqref{imwidth}, we infer that $|x_{k-1}(E+i\varepsilon)| > 2 + \delta$ and $|x_{k}(E+i\varepsilon)| > 2 +
\delta$. Since $|x_{-1}(E+i\varepsilon)| = 2 \le 2 + \delta$, the condition \eqref{e.escape.criteria} holds for some $k_0 \le
k$. Lemma~\ref{l.complexenergyBISTlemma} implies that $|x_k(z)| \ge (1 + \delta)^{G^{(k_0)}_{k - k_0}} + 1$ for $k \ge k_0$.
In particular, for $k' \ge k$, we must have
$$
|x_{k'}(E+i\varepsilon)| \ge (1 + \delta)^{G^{(k_0)}_{k'- k}}.
$$
Moreover, notice that $G^{(k_0)}_j$ always satisfies a uniform exponential (in $j$) lower bound.

This motivates the following definitions. Fix some small $\delta > 0$. For $T > 1$, denote by $k(T)$ the unique integer with
$$
\frac{q_{k(T) - 2}^{s'}}{C_\delta'} \le T < \frac{q_{k(T) - 1}^{s'}}{C_\delta'}
$$
and let
$$
N(T) = q_{k(T) + \lfloor \sqrt{k(T)} \rfloor}.
$$
Note that
\begin{align*}
\limsup_{T \to \infty} \frac{\log N(T)}{\log T} & = \limsup_{T \to \infty} \frac{\log q_{k(T) + \lfloor \sqrt{k(T)}
\rfloor}}{\log T} \\ & = \limsup_{T \to \infty} \frac{\log q_{k(T) + \lfloor \sqrt{k(T)} \rfloor}}{k(T) + \lfloor \sqrt{k(T)}
\rfloor} \frac{k(T) + \lfloor \sqrt{k(T)} \rfloor}{\log T} \\ & \le \limsup_{T \to \infty} \frac{\log q_{k(T) + \lfloor
\sqrt{k(T)} \rfloor}}{k(T) + \lfloor \sqrt{k(T)} \rfloor} \frac{k(T) + \lfloor \sqrt{k(T)} \rfloor}{s' \log q_{k(T) - 2}} \\ &
= \frac{1}{s'} \limsup_{T \to \infty} \frac{\log q_{k(T) + \lfloor \sqrt{k(T)} \rfloor}}{k(T) + \lfloor \sqrt{k(T)} \rfloor}
\frac{k(T) - 2}{\log q_{k(T) - 2}} \\ & = \frac{1}{s'},
\end{align*}
since $\lim_{k \to \infty} \frac1k \log q_k$ exists and is finite. Similarly, we see that
$$
\liminf_{T \to \infty} \frac{\log N(T)}{\log T} \ge \frac{1}{s'}.
$$
Thus, for every $\tilde \nu > 0$, there are constants $C_{\tilde \nu,1}, C_{\tilde \nu,2} > 0$ such that
\begin{equation}\label{ntfib}
C_{\tilde \nu, 1} T^{\frac{1}{s'} - \tilde \nu} \le N(T) \le C_{\tilde \nu, 2} T^{\frac{1}{s'} + \tilde \nu}.
\end{equation}
It follows from \cite[Theorem~7]{DT07} and the argument above that\footnote{This estimate obviously works for $\omega = 0$
since then the trace and the norm are directly related. For general $\omega$, one can use the arguments developed in
\cite{D05}. The central idea is that the trace of words of length $q_k$ occurring in Sturmian sequences of slope $\alpha$ is
the same for all but one word and is given by $x_k$. If the word in question is the ``bad'' one, we can simply shift by one to
see a good word, derive the estimate there and divide by $C^2$, where $C$ bounds the norm of a one-step transfer matrix.}
\begin{align}
\langle P(N(T),\cdot) \rangle (T) & \lesssim \exp (-c N(T)) + T^3 \int_{-K}^K \left( \max_{3 \le n \le N(T)} \left\| M \left(
n; \omega, E+\tfrac{i}{T} \right) \right\|^2 \right) ^{-1} dE \nonumber \\ & \lesssim \exp (-c N(T)) + T^3 (1 + \delta)^{-2
G^{(k_0)}_{\lfloor \sqrt{k(T)} \rfloor}}. \label{e.pntestimate}
\end{align}
(We can estimate the norm on the left half-line in a completely analogous way.) From this bound, we see that $\langle
P(N(T),\cdot) \rangle (T)$ goes to zero faster than any inverse power of $T$. Indeed, due to \eqref{ntfib} this is clear for
the first term \eqref{e.pntestimate}, and for the second term in \eqref{e.pntestimate}, we note that for any $m > 0$, we have
\begin{align*}
\limsup_{T \to \infty} & \log \left[ T^m (1 + \delta)^{-2 G^{(k_0)}_{\lfloor \sqrt{k(T)} \rfloor}} \right] \le \limsup_{T \to
\infty} \log \left[ (N(T))^{\tilde m} (1 + \delta)^{-2 G^{(k_0)}_{\lfloor \sqrt{k(T)} \rfloor}} \right] \\ & \le \limsup_{T
\to \infty} \log \left[ (q_{k(T) + \lfloor \sqrt{k(T)} \rfloor})^{\tilde m} (1 + \delta)^{-2 G^{(k_0)}_{\lfloor \sqrt{k(T)}
\rfloor}} \right] \\ & = \limsup_{k \to \infty} \log \left[ (q_{k + \lfloor \sqrt{k} \rfloor})^{\tilde m} (1 + \delta)^{-2
G^{(k_0)}_{\lfloor \sqrt{k} \rfloor}} \right] \\ & = \limsup_{k \to \infty} \left[ \tilde m \log (q_{k + \lfloor \sqrt{k}
\rfloor}) - 2 G^{(k_0)}_{\lfloor \sqrt{k} \rfloor} \log (1 + \delta) \right] \\ & \le \limsup_{k \to \infty} \left[ C_1 \tilde
m (k + \lfloor \sqrt{k} \rfloor) - C_2^{\sqrt{k}} \log (1 + \delta) \right] \\ & = - \infty
\end{align*}
for a suitable $C_2 > 1$. Here we used \eqref{ntfib}, the existence and finiteness of $\lim_{k \to \infty} \frac1k \log q_k$,
and a trivial lower bound for $G^{(k_0)}$.

Therefore we can apply \cite[Theorem~1]{DT07} and obtain from \eqref{ntfib} that
$$
\tilde \alpha_u^+ \le \frac{1}{s'} + \tilde \nu = \left( \frac{\liminf_{k \to \infty} \frac1k \log
\frac{1}{R_k(\delta)}}{\lim_{k \to \infty} \frac1k \log q_k} - \rho' \right)^{-1} + \tilde \nu.
$$
Since we can take $\rho' > 0$ and $\tilde \nu > 0$ arbitrarily small, \eqref{e.alphauupperbound} follows.

\bigskip

Let us now show \eqref{e.alphaupluslowerbound} and \eqref{e.alphaulowerbound}. Assume first that $\alpha$ is a bounded density number. Then, $\liminf_{k \to \infty} \frac1k \log q_k$ is finite. If $\limsup_{k \to \infty} \frac1k \log \frac{1}{r_k(\delta)}$ is infinite, there is nothing to prove, so let us consider the case where $\limsup_{k \to \infty} \frac1k \log \frac{1}{r_k(\delta)}$ is finite.

Consider $\delta \in (0,\delta(\lambda))$, $\varepsilon > 0$, and a $B \in \mathcal{G}_k$ that serves as a maximizer in
$r_k(\delta) = \max_{B \in \mathcal{G}_k} \sup \{ r > 0 : B(z_B , r) \subseteq B^\delta \}$. Recall that $z_B$ is the unique
zero of $t_B$ in $B^\delta$.

For $\rho > 0$ arbitrary, let
\begin{equation}\label{e.sdef}
s = \frac{\limsup_{k \to \infty} \frac1k \log \frac{1}{r_k(\delta)}}{\liminf_{k \to \infty} \frac1k \log q_k} + \rho.
\end{equation}
Clearly, $s$ is strictly positive. By definition of $s$, for suitably chosen $C_\delta > 0$, we have
\begin{equation}\label{e.sprop}
C_\delta q_k^{s} \ge \frac{2}{r_k(\delta)}
\end{equation}
for every $k \ge 0$.

Take $N = q_k$ and consider $T \ge C_\delta N^{s}$ (which in turn implies $T \ge \frac{2}{r_k(\delta)}$ by \eqref{e.sprop}).
Due to the Parseval formula \eqref{e.parsform}, we can bound the time-averaged outside probabilities from below as follows,
\begin{equation}\label{parseval}
\langle P(N,\cdot) \rangle (T) \gtrsim \frac1T \int_\R \left( \max \left\{ \|M(N;\lambda,\alpha,\omega,E+\tfrac{i}{T})\|,
\|M(-N;\lambda,\alpha,\omega,E+\tfrac{i}{T})\| \right\} \right)^{-2} \, dE.
\end{equation}
See, for example, the proof of \cite[Theorem~1]{DT03} for an explicit derivation of \eqref{parseval} from \eqref{e.parsform}.

By Lemma~\ref{l.powerlawbound} there are constants $C,\xi$ such that for every $k$, every $z \in B^\delta$, and every $\omega
\in \T$, we have
\begin{equation}\label{e.2sidedpowerlaw}
\| M(n;\lambda,\alpha,\omega,z) \| \le C |n|^\xi.
\end{equation}
for $1 \le |n| \le q_k$.

To bound the integral from below, we integrate only over those $E \in (z_B-r_k(\delta), z_B+r_k(\delta))$ for which $E+i/T \in
B(z_B, r_k(\delta)) \subset B^\delta$. Since $\frac{1}{T} \le \frac{r_k(\delta)}{2}$, the length of such an interval $I_k$ is
larger than $cr_k(\delta)$ for some suitable $c > 0$. For $E \in I_k$, we have
$$
\|M(N;\lambda,\alpha,\omega, E+i\varepsilon)\| \lesssim N^{\xi} \lesssim T^{\frac{\xi}{s}}.
$$
Therefore, \eqref{parseval} together with \eqref{e.2sidedpowerlaw} gives
\begin{equation}\label{eq1}
\langle P(N,\cdot) \rangle (T) \gtrsim \frac{r_k}{T} \, T^{-\frac{2\xi}{s}} \gtrsim T^{-2-\frac{2\xi}{s}},
\end{equation}
where $N = q_k$, $T \ge C_\delta N^s$, for any $k \ge k_0$.

In particular, for $T_k = C_\delta q_k^s$, we obtain
$$
\left\langle P \left( \tfrac{1}{C_\delta^{1/s}} T_k^{\frac{1}{s}},\cdot \right) \right\rangle (T_k) = \langle P(q_k,\cdot)
\rangle (T_k) \gtrsim T_k^{-2-\frac{2\xi}{s}},
$$
which implies that
$$
\tilde \beta^+(p) \ge \frac{1}{s} - \frac{2}{p}\left( 1 + \frac{\xi}{s} \right)
$$
and
$$
\tilde \alpha_u^+ \ge \frac{1}{s} = \left( \frac{\limsup_{k \to \infty} \frac1k \log \frac{1}{r_k(\delta)}}{\liminf_{k \to
\infty} \frac1k \log q_k} + \rho \right)^{-1}.
$$
Since $\rho > 0$ can be taken arbitrarily small, this proves \eqref{e.alphaupluslowerbound}.\footnote{The reader may be
concerned about this lower bound assuming that it is possible for the right-hand side to exceed $1$, as it is known that
$\tilde \alpha_u^+$ cannot exceed $1$. However, there is no need to worry as we must have $\limsup_{k \to \infty} \frac1k \log
\frac{1}{r_k(\delta)} \ge \liminf_{k \to \infty} \frac1k \log q_k$. This is simply because of the number of connected
components in question and the fact that their intersections with the real line are all contained in a fixed
($\lambda$-dependent, but $k$-independent) compact subset of $\R$.}

To prove \eqref{e.alphaulowerbound}, assume the stronger condition that $\alpha$ is of bounded type. Let us take any sufficiently large $T$ and choose $k$ maximal with $C_\delta q_k^s \le T$. Then,
$$
C_\delta q_k^s \le T < C_\delta q_{k+1}^s \le B^s C_\delta q_k^s,
$$
where $B$ is chosen so that $q_{k+1} \le B q_k$ for every $k$. Such a $B$ exists since $\alpha$ is of bounded type.

It follows from \eqref{eq1} that
$$
\left\langle P \left( \tfrac{1}{B C_\delta^{1/s}} T^{\frac{1}{s}},\cdot \right) \right\rangle (T) \ge \langle P(q_k,\cdot)
\rangle (T) \gtrsim T^{-2-\frac{2\xi}{s}}
$$
for all sufficiently large $T$. It follows from the definition of $\tilde \beta^-(p)$ and $\tilde \alpha_u^-$ that
$$
\tilde \beta^-(p) \ge \frac{1}{s} - \frac{2}{p}\left( 1 + \frac{\xi}{s} \right)
$$
and
$$
\tilde \alpha_u^- \ge \frac{1}{s} = \left( \frac{\limsup_{k \to \infty} \frac1k \log \frac{1}{r_k(\delta)}}{\liminf_{k \to
\infty} \frac1k \log q_k} + \rho \right)^{-1},
$$
by \eqref{e.sdef}. Since $\rho > 0$ can be taken arbitrarily small, this proves \eqref{e.alphaulowerbound}.

(c) The estimates in this part follow immediately from the estimates in parts (a) and (b). This concludes the proof.
\end{proof}

\begin{remark}\label{r.marinissues}
Let us comment on the results and proofs in \cite{M10}. Unfortunately, there are several issues. We have already pointed out
that \cite[Lemma~2]{M10} is incorrect as stated. Moreover, in the proofs of the main theorems of \cite{M10} there are several
mistakes and gaps. In the proof of \cite[Theorem~1]{M10}, \cite[Lemma~2]{M10} is claimed to be applicable on
\cite[p.~871]{M10}. However, the justification for this is not sufficient. Indeed it does not follow from the arguments given
there that \eqref{e.escape} holds for some $k_0$. Next, there is a mistake in the chain of inequalities at the bottom of
\cite[p.~871]{M10}. The inequality goes in the wrong direction because one has the opposite of what is needed due to the
definition of $\gamma(V)$ on \cite[p.~871]{M10}. This strongly calls into question that the strategy of the proof can work
under the assumption of \cite[Theorem~1]{M10}, that is, when merely assuming that $\limsup_{k \to \infty} \frac1k \log q_k$ is
finite. For these reasons, neither is \cite[Theorem~1]{M10} completely proved in \cite{M10}, nor is it at all clear that
\cite[Theorem~1]{M10} is correct as stated. Finally, as we will point out in the last section, the proof of
\cite[Theorem~2]{M10} is flawed as well and it is also not at all clear whether \cite[Theorem~2]{M10} is correct as stated.
\end{remark}

\section{Asymptotics of Transport Exponents for Sturmian Potentials of Constant Type}\label{sec.5}

In this section we prove Theorem \ref{t.consttype}. The theorem will follow quickly from Propositions~\ref{p.maxband} and
\ref{p.transportbounds}.

\begin{proof}[Proof of Theorem \ref{t.consttype}]
The statement is a quantitative version of \eqref{e.exponentsareequal}. Indeed, $\frac{m+\sqrt{m^2+4}}{2} $ is the larger
eigenvalue of the matrix $\begin{pmatrix}
                                                                                                                                     m
                                                                                                                                     &
                                                                                                                                     1
                                                                                                                                     \\
                                                                                                                                     1
                                                                                                                                     &
                                                                                                                                     0
                                                                                                                                     \\
                                                                                                                                   \end{pmatrix}$,
                                                                                                                                   hence
\begin{equation}\label{e.qk}
\lim_{k \to \infty} \frac1k \log q_k=\log \frac{m+\sqrt{m^2+4}}{2}.
\end{equation}
The limit $\lim_{k \to \infty} \frac{1}{k} \log \min \{ |t_B'(z_B)| : B \in \mathcal{G}_k \}$ does exist (and in fact can be
described in terms of the multipliers of the periodic points of the trace map that corresponds to this potential). For the
case of Fibonacci Hamiltonian, this was proven in \cite[Proposition~3.7]{DGY14}, but the proof for the case $m \ne 1$ is a
verbatim repetition. Moreover, there is a direct connection between $\min \{ |t_B'(z_B)| : B \in \mathcal{G}_k \}$ and the
size of the maximal band $|B_{max, k}|$:

\begin{prop}[Corollary 3.2 from \cite{LQW14}]\label{p.lqw14}
Let $\lambda \ge 20$ and $\alpha$ be irrational. Then there exists a constant $\xi = 4 \exp (180\lambda) > 1$ such that for
any spectral generating band $B$ with generating polynomial $t_B$,
$$
\xi^{-1}\le |t_B'(E)|\cdot |B| \le \xi, \ \ \forall E\in B.
$$
\end{prop}

This implies that
\begin{equation}\label{e. bmax}
\lim_{k \to \infty} \frac{1}{k} \log \min \{ |t_B'(z_B)| : B \in \mathcal{G}_k \} = -\lim_{k \to \infty} \frac{1}{k} \log
|B_{max, k}|.
\end{equation}

Therefore, we can use Proposition~\ref{p.maxband} to estimate $\lim_{k \to \infty} \frac{1}{k} \log \min \{ |t_B'(z_B)| : B
\in \mathcal{G}_k \}$. We have
\begin{align*}
-\log 8 - \log m + & \left( -1 + \frac{1}{k} \sum_{j=1}^s \left\lfloor \frac{m_j+1}{2} \right\rfloor \right) \log \lambda \le
\frac{1}{k} \log |B_{max, k}| \\ & \le \log 48 - \log m + \left( -1 + \frac{1}{k} \sum_{j=1}^s \left\lfloor \frac{m_j+1}{2}
\right\rfloor \right) \log \lambda.
\end{align*}
In the case when $m=1$ we have $s=1$ and $m_1=k$, hence we get
\begin{equation}\label{e.m1}
-\log 8 + \frac{1}{k} \left( \left\lfloor \frac{k+1}{2} \right\rfloor - 1 \right) \log \lambda \le \frac{1}{k} \log |B_{max,
k}| \le \log 48 + \frac{1}{k} \left( \left\lfloor \frac{k+1}{2} \right\rfloor - 1 \right) \log \lambda.
\end{equation}
In the case when $m>1$ we have $s = 0$, and hence
\begin{equation}\label{e.mnot1}
-\log 8 - \log m - \log \lambda \le \frac{1}{k} \log |B_{max, k}| \le \log 48 - \log m - \log \lambda.
\end{equation}
Now Theorem \ref{t.consttype} follows from a combination of \eqref{e.exponentsareequal}, \eqref{e.qk}, \eqref{e. bmax},
\eqref{e.m1}, and \eqref{e.mnot1}.
\end{proof}

%\begin{remark} The existence of the limit $\lim_{k \to \infty} \frac{1}{k} \log \min \{ |t_B'(z_B)| : B \in \mathcal{G}_k \}$
holds not only for Sturmian operators with frequencies of constant type, but also for frequencies with periodic (or just
eventually periodic) continued fraction expansion. These operators can be studied using the trace map formalism \cite{G14,
Me14}, and the proof from \cite{DGY14} of the existence of the limit works without any changes. The explicit asymptotics can
also be calculated in each particular case. For example, for $\alpha = [1,2,1,2,1,2,\ldots]$, we have $\lim_{\lambda \to
\infty} \tilde{\alpha}^{\pm}_u \cdot \log \lambda =  \log (2 + \sqrt{3})$. And for $\alpha' = [1, 1, 2, 2, 1, 1, 2, 2, 1, 1,
2, 2,\ldots]$, we have $\lim_{\lambda \to \infty} \tilde{\alpha}^{\pm}_u \cdot \log \lambda = \frac{1}{3} \log \frac{15 +
\sqrt{221}}{2}$. Notice that the frequencies of each coefficient in the continued fraction expansion of $\alpha$ and $\alpha'$
in these examples are the same, while the transport exponents are different (for sufficiently large coupling). At the same
time, \eqref{e.exponentsareequal} together with Proposition~\ref{p.maxband} imply for irrational numbers of bounded type that
if \eqref{e.exponentsareequal} holds, then the corresponding asymptotics  $\lim_{\lambda \to \infty} \tilde{\alpha}^{\pm}_u
\cdot \log \lambda$ is a ``tail property'' of the sequence of continued fraction coefficients. Notice that it is known that
the Hausdorff dimension of the spectrum of Sturmian Hamiltonian is a ``tail property'' of the sequence of continued fraction
coefficients \cite{DG15} (for sufficiently large coupling). This motivates the following question.

\begin{quest}
Is it true that the transport exponents of a Sturmian Hamiltonian depend only on the ``tail'' of the continued fraction
expansion of the corresponding irrational number?
\end{quest}

\section{The Transport Exponent for a Lebesgue Typical Frequency}

In this section we prove Theorem~\ref{t.bound}. We first discuss the asymptotic behavior of the length of the longest spectral
generating band.

\begin{prop}\label{p.ae}
Denote by $B_{max, k}$ the largest band in $\mathcal{G}_k$. For Lebesgue almost every $\alpha \in \R/\Z$, we have
$$
\lim_{\lambda \to \infty} \frac{1}{\log \lambda} \left( \limsup_{k \to \infty} \frac1k \log |B_{max, k}| \right) =
\lim_{\lambda \to \infty} \frac{1}{\log \lambda} \left( \liminf_{k \to \infty} \frac1k \log |B_{max, k}| \right) =-\frac{\log
\frac{1+\sqrt{5}}{2}}{\log 2}.
$$
\end{prop}

\begin{proof}
We apply Proposition~\ref{p.maxband}. Note first that for Lebesgue almost every $\alpha$, $\lim_{k \to \infty} \frac1k \log
\delta_k^{-k}$ exists and is finite.\footnote{Indeed, for Lebesgue almost every $\alpha$, $\lim_{k \to \infty}  \delta_k$
exists and is equal to the Khinchin constant $K_0 =
2.685452001065306445309714835481795693820382293994462953051152345557218\ldots$. We are not using the specific value of the
Khinchin constant here.} Since we are going to divide by $\log \lambda$ and send $\lambda \to \infty$, we can concentrate on
the third factor in
\begin{equation}\label{e.bmaxestimate}
8^{-k} \cdot \delta_k^{-k} \cdot \lambda^{-k+\sum_{j=1}^s \lfloor (m_j+1)/2 \rfloor} \le |B_{max, k}| \le 48^k \cdot
\delta_k^{-k} \cdot \lambda^{-k+\sum_{j=1}^s \lfloor (m_j+1)/2 \rfloor},
\end{equation}
which holds for $\lambda \ge 20$; compare \eqref{main-esti}.

We have
\begin{align*}
\sum_{j=1}^s \left\lfloor \frac{m_j+1}{2} \right\rfloor & = \sum_{\substack{1 \le j \le s\\m_j \text{ odd}}} \frac{m_j+1}{2} +
\sum_{\substack{1 \le j \le s\\m_j \text{ even}}} \frac{m_j}{2} \\ & = \sum_{1 \le j \le s} \frac{m_j}{2} + \frac{1}{2} \# \{
1 \le j \le s : m_j \text{ odd} \}.
\end{align*}

Due to the ergodicity of the Gauss measure we have, in the notation from Proposition~\ref{p.maxband}, for Lebesgue almost
every $\alpha$,
\begin{align*}
\lim_{k \to \infty} \frac{1}{2k} \sum_{j=1}^s m_j & = \lim_{k \to \infty} \frac{1}{2k} \# \{ \text{$1$'s in } a_1 \ldots a_k
\} \\ & = \frac12 \cdot \frac{1}{\log 2} \int_{1/2}^{1}\frac{dx}{1+x} \\ & = \frac{1}{2\log 2} \left[ \log 2 - \log
\frac{3}{2} \right] \\ & = \frac{\log 4/3}{2\log 2}.
\end{align*}
For $x\in [0,1]$, let $[a_1(x),a_2(x),\cdots]$ be its continued fraction expansion,  write $w_n(x)=a_1(x)\cdots a_n(x)$.
Define Fibonacci sequence as
$$
F_0=1, F_1=1, F_{k+1}=F_k+F_{k-1} \ \ (k\ge 1).
$$
Write $\theta_k=F_{k-1}/F_k.$ By induction  we can show that
$$
\begin{cases}
w_{2p-1}(x)=1^{2p-1} \ \ &\text{ if and only if }\ \  x\in[\theta_{2p},\theta_{2p-1}], \\ w_{2p}(x)=1^{2p} \ \ &\text{ if and
only if }\ \  x\in[\theta_{2p},\theta_{2p+1}].
\end{cases}
$$
Therefore, we have
\begin{align*}
\int_{[0,1]}\chi_{[1^{m}]}(w_{m}(x)) \frac{dx}{1+x}=
\begin{cases}
\log\frac{(1+\theta_{2p-1})}{(1+\theta_{2p})}& m=2p-1, \\ \log\frac{(1+\theta_{2p+1})}{(1+\theta_{2p})}& m=2p.
\end{cases}
\end{align*}
Denote $\theta_\infty=\lim_{p\to \infty}\theta_p=\frac{\sqrt{5}-1}{2}$.

Thus,
\begin{align*}
\lim_{k \to \infty} & \frac{1}{2k} \# \{ 1 \le j \le s : m_j \text{ odd} \}\\
  =& \frac{1}{2\log 2}\int_{[0,1]}\sum_{p\ge1}\sum_{m\ge 2}\sum_{n\ge 2}\chi_{[m1^{2p-1}n]}(w_{2p+1}(x))\frac{dx}{1+x}\\
  =&\sum_{p\ge1} \sum_{m\ge 2}\frac{1}{2\log 2}\int_{[0,1]}\left(\sum_{n\ge
  1}\chi_{[m1^{2p-1}n]}(w_{2p+1}(x))-\chi_{[m1^{2p}]}(w_{2p+1}(x))\right)\frac{dx}{1+x}\\
  =&\sum_{p\ge1} \sum_{m\ge 2}\frac{1}{2\log
  2}\int_{[0,1]}\left(\chi_{[m1^{2p-1}]}(w_{2p}(x))-\chi_{[m1^{2p}]}(w_{2p+1}(x))\right)\frac{dx}{1+x}\\
  =&\sum_{p\ge1} \sum_{m\ge 1}\frac{1}{2\log
  2}\int_{[0,1]}\left(\chi_{[m1^{2p-1}]}(w_{2p}(x))-\chi_{[m1^{2p}]}(w_{2p+1}(x))\right)\frac{dx}{1+x}\\
  &-\sum_{p\ge1} \frac{1}{2\log
  2}\int_{[0,1]}\left(\chi_{[1^{2p}]}(w_{2p}(x))-\chi_{[1^{2p+1}]}(w_{2p+1}(x))\right)\frac{dx}{1+x}\\
   =&\sum_{p\ge1}\frac{1}{2\log
   2}\int_{[0,1]}\left(\chi_{[1^{2p-1}]}(w_{2p-1}(x))-\chi_{[1^{2p}]}(w_{2p}(x))\right)\frac{dx}{1+x}\\
  &-\sum_{p\ge1} \frac{1}{2\log
  2}\int_{[0,1]}\left(\chi_{[1^{2p}]}(w_{2p}(x))-\chi_{[1^{2p+1}]}(w_{2p+1}(x))\right)\frac{dx}{1+x}\\
  =&\sum_{p\ge1}\frac{1}{2\log
  2}\int_{[0,1]}\left(\chi_{[1^{2p-1}]}(w_{2p-1}(x))-2\chi_{[1^{2p}]}(w_{2p}(x))+\chi_{[1^{2p+1}]}(w_{2p+1}(x))\right)\frac{dx}{1+x}\\
  =&\sum_{p\ge1}\frac{1}{2\log
  2}\left(\log\frac{1+\theta_{2p-1}}{1+\theta_{2p}}-2\log\frac{1+\theta_{2p+1}}{1+\theta_{2p}}+\log\frac{1+\theta_{2p+1}}{1+\theta_{2p+2}}\right)\\
  =&\sum_{p\ge1}\frac{1}{2\log 2}\log\frac{(1+\theta_{2p-1})(1+\theta_{2p})}{(1+\theta_{2p+1})(1+\theta_{2p+2})} \\
   =&\frac{1}{2\log 2}\log\frac{(1+\theta_{1})(1+\theta_{2})}{(1+\theta_{\infty})^2} \\
   =&-\frac{\log \frac{3+\sqrt{5}}{6}}{2\log 2},
\end{align*}
where for the fifth equality we use the fact that the Gauss measure is invariant. Therefore,
\begin{align*}
-1 + \lim_{k \to \infty} \frac{1}{k} \sum_{j=1}^s \left\lfloor \frac{m_j+1}{2} \right\rfloor & = -1 + \lim_{k \to \infty}
\frac{1}{2k} \sum_{j=1}^s \left( m_j + \# \{ 1 \le j \le s : m_j \text{ odd} \} \right)\\ & = -1 + \frac{\log 4/3}{2\log 2}
-\frac{\log \frac{3+\sqrt{5}}{6}}{2\log 2} \\ & = -\frac{\log \frac{1+\sqrt{5}}{2}}{\log 2},
\end{align*}
and the proposition follows from this statement together with \eqref{e.bmaxestimate}.
\end{proof}

\begin{proof}[Proof of Theorem~\ref{t.bound}.]
By \eqref{e.upperbound} we have
$$
\tilde \alpha_u^+ \le \frac{\lim_{k \to \infty} \frac1k \log q_k}{\liminf_{k \to \infty} \frac{1}{k} \log \min \{ |t_B'(z_B)|
: B \in \mathcal{G}_k \}}
$$
for $\lambda > 4$, provided that $\alpha$ is such that $\lim_{k \to \infty} \frac1k \log q_k$ exists and is finite.

Applying Proposition~\ref{p.lqw14}, we get
$$
\tilde \alpha_u^+ \le - \frac{\lim_{k \to \infty} \frac1k \log q_k}{\limsup_{k \to \infty} \frac{1}{k} \log \left( |B_{max,
k}| \right)}
$$
for $\lambda \ge 20$.

It is well known that for Lebesgue almost every irrational number $\alpha$, one has $\lim_{k \to \infty} \frac1k \log q_k =
\frac{\pi^2}{12\log 2}$. Therefore, as a consequence of \eqref{e.upperbound} and Proposition~\ref{p.ae}, the result follows.
\end{proof}

\section{Quasi-Ballistic Transport for a Generic Set of Frequencies}

In this section we prove Theorem~\ref{t.quasiballistic2}. That is, for arbitrary coupling and for frequencies $\alpha$ from a
dense $G_\delta$ subset of $(0,1)$, we have quasi-ballistic transport in the sense that all the transport exponents associated
with sequences of time scales are equal to one. The proof will be inspired by a construction of Last \cite{L96} in the case of
the almost Mathieu operator.

\begin{proof}[Proof of Theorem~\ref{t.quasiballistic2}.]
Fix $\lambda > 0$. Given some rational number $\alpha_r = [a_1(\alpha_r), \ldots, a_{\ell_{\alpha_r}}(\alpha_r)]$ in $(0,1)$,
we consider the periodic operator $H_{\lambda,\alpha_r,0}$. This operator has ballistic transport due to the periodicity of
its potential, and hence operators with potentials that are sufficiently close to this potential on a sufficiently large
interval around the support of the initial state have similar transport behavior, at least as long as the evolution is
essentially confined (by the general ballistic upper bound) to the interval on which the potentials are close. Since our
potentials take only two values, we will actually want to enforce coincidence on the large intervals in question. By the
hierarchical structures of Sturmian potentials, if we consider an $\alpha \in (0,1)$ with continued fraction coefficients $\{
a_j(\alpha) \}$ so that $a_j(\alpha) = a_j(\alpha_r)$ for $1 \le j \le \ell_{\alpha_r}$ and $a_{\ell_{\alpha_r} + 1}(\alpha)$
large, then, on the one hand, $\alpha$ and $\alpha_r$ are close and, on the other hand, the potential of
$H_{\lambda,\alpha,0}$ coincides with that of $H_{\lambda,\alpha_r,0}$ on $[1,q_{\ell_{\alpha_r}}]$ and, moreover, this block
is repeated many (namely $a_{\ell_{\alpha_r} + 1}(\alpha)$) times. Now we need to shift these $a_{\ell_{\alpha_r} +
1}(\alpha)$ blocks to the left so that they are centered around the origin, since we are studying the evolution of
$\delta_0$.\footnote{This is the point missed in the proof of \cite[Theorem~6]{M10}.} This can be accomplished by choosing a
suitable phase, that is, by considering the operator $H_{\lambda,\alpha,\omega}$ for a suitable $\omega \in \R/\Z$.

The ideas above can be turned into a quantitative statement. Choose a sequence $\{ p_m \}_{m \in \Z_+} \subset (0,1]$ such
that each number $1/q$, $q \in \Z_+$ appears in this sequence infinitely many times. Then, a minor generalization of
\cite[Lemma~7.2]{L96} and some ideas from the proof of \cite[Lemma~4.2]{DLY} yield the following: For every $m \in \Z_+$,
there exists $S(m,\alpha_r) \in \Z_+$, which can in addition be assumed to satisfy $S(m,\alpha_r) \ge m$, such that if
$a_{\ell_{\alpha_r} + 1}(\alpha) \ge S(m,\alpha_r)$, then for a suitable choice of $\omega \in \R/\Z$, we have that
\begin{equation}\label{e.transportestimate}
\langle \langle |X|_{\delta_0}^{p_m} \rangle \rangle (T_m) > \frac{T_m^{p_m}}{\log T_m} \quad \text{ for some } T_m \ge m.
\end{equation}
The operator implicit in this statement is $H_{\lambda,\alpha,\omega}$. In fact, a bit more can be said, and this will become
important shortly. The potential corresponding to $\alpha_r$ is $q_{\ell_{\alpha_r}}(\alpha_r)$-periodic, and any Sturmian
potential with frequency $\alpha$ whose continued fraction coefficients coincide with those of $\alpha_r$ up to index
$\ell_{\alpha_r}$ have this periodic block present, and actually repeated $a_{\ell_{\alpha_r} + 1}(\alpha)$ many times; see
the papers \cite{DL99a, DL99} on partitions of Sturmian sequences for a detailed study of this structure. Last's argument from
\cite{L96} derives \eqref{e.transportestimate} for all potentials that coincide with the periodic reference potential on a
sufficiently large interval around the origin. This shows that if $a_{\ell_{\alpha_r} + 1}(\alpha)$ is large enough, a
suitable shift forces coincidence on the needed interval size. The point now is that in our iteration below of this argument,
the potential will not be changed anymore on this interval, and hence the estimate \eqref{e.transportestimate} remains valid
for all the subsequent modifications of the potential (which turns a sequence of periodic potentials, corresponding to a
sequence of rational $\alpha_r$'s into an aperiodic limit sequence).

Set
$$
U(m,\alpha_r) := \left\{ \alpha : a_j(\alpha) = a_j(\alpha_r) \text{ for } 1 \le j \le \ell_{\alpha_r}, \; a_{\ell_{\alpha_r}
+ 1}(\alpha) \ge S(m,\alpha_r) \right\}.
$$
This is an open subset of $(0,1)$. Letting $\alpha_r$ range over all rational numbers in $(0,1)$, we obtain the open and dense
set
$$
\bigcup_{\alpha_r} U(m,\alpha_r).
$$
Thus, the set
$$
\mathcal{G} := \bigcap_{m \in \Z_+} \bigcup_{\alpha_r} U(m,\alpha_r)
$$
is a dense $G_\delta$ set.

Consider $\alpha \in \mathcal{G}$. Then, for every $m \in \Z_+$, there exists a rational number $\alpha_r = \alpha_r(m)$ such
that $\alpha \in U(m,\alpha_r)$. Consider the sequence $\alpha_r = \alpha_r(m)$, $m \to \infty$, such that $\alpha \in
U(m,\alpha_r)$. In particular, we have $a_j(\alpha) = a_j(\alpha_r) \text{ for } 1 \le j \le \ell_{\alpha_r(m)}$. There are
two cases: either $\ell_{\alpha_r(m)}$ is bounded as $m \to \infty$, or not. Since $a_{\ell_{\alpha_r(m)} + 1}(\alpha) \ge
S(m,\alpha_r) \ge m$, the first case is impossible. Thus, $\ell_{\alpha_r(m)}$ is unbounded as $m \to \infty$. But then, by
the coincidence $a_j(\alpha) = a_j(\alpha_r) \text{ for } 1 \le j \le \ell_{\alpha_r(m)}$, the $\alpha_r(m)$ can be ordered in
such a way that we only add new continued fraction coefficients in each step. Each time $\alpha_r = [a_1(\alpha_r), \ldots,
a_{\ell_{\alpha_r}}(\alpha_r)]$ is extended to $\alpha_{r'} = [a_1(\alpha_r), \ldots, a_{\ell_{\alpha_r}}(\alpha_r),
a_{\ell_{\alpha_r}+1}(\alpha_{r'}), \ldots, , a_{\ell_{\alpha_{r'}}}(\alpha_{r'})]$, we enlarge the period and the new
periodic potential coincides with the previous one on a large number of the previous periods in both directions from the
origin, so that the estimate \eqref{e.transportestimate} holds, and will continue to hold if we modify the new periodic
potential in a similar way to one with an even larger period.

For the limit potential, which by construction belongs to the Sturmian subshift contained in $\{ 0,\lambda \}^\Z$ of slope
$\alpha$, we have the estimate \eqref{e.transportestimate} for every $m$. Since the sequence $\{ p_m \}_{m \in \Z_+} \subset
(0,1]$ contains each number $1/q$, $q \in \Z_+$ infinitely many times, it follows that $\tilde \beta^+(1/q) = 1$ for every $q
\in \Z_+$. By the fact that $\tilde \beta^+(p)$ is non-decreasing in $p$ it follows that, in fact, we have $\tilde \beta^+(p)
= 1$ for every $p > 0$. The desired genericity statement follows.
\end{proof}

\end{document}